\documentclass{amsart}

\usepackage{amscd}
\usepackage{amssymb}

\usepackage{color}

\newtheorem{thm}[equation]{Theorem}
\newtheorem{lem}[equation]{Lemma}

\newtheorem{cor}[equation]{Corollary}
\newtheorem{prop}[equation]{Proposition}
\newtheorem{conj}[equation]{Conjecture}

\newtheorem{specthm}{Theorem}

\newtheorem{speccor}[specthm]{Corollary}

\newtheorem*{thm*}{Theorem}
\newtheorem*{prop*}{Proposition}
\newtheorem*{cor*}{Corollary}
\newtheorem*{lem*}{Lemma}
\newtheorem*{MT*}{Main Theorem}

\newtheorem*{ques*}{Question}

\theoremstyle{definition} %

\newtheorem*{defn*}{Definition}

\newtheorem{eg}[equation]{Example}

\theoremstyle{remark} %
\newtheorem{rmk}[equation]{Remark}

\newtheorem*{rmk*}{Remark}
\newtheorem*{rmks*}{Remarks}

\DeclareMathOperator{\Nil}{Nil}
\DeclareMathOperator{\rank}{rank}
\DeclareMathOperator{\type}{type}

\DeclareMathOperator{\car}{char}
\DeclareMathOperator{\Lie}{Lie}

\DeclareMathOperator{\im}{im}

\renewcommand{\P}{\mathbb{P}}
\newcommand{\C}{\mathbb{C}}
\newcommand{\Z}{\mathbb{Z}}

\newcommand{\lie}{\mathfrak{g}}
\newcommand{\lsub}{\mathfrak{h}}

\newcommand{\qform}[1]{{\left\langle{#1}\right\rangle}}                   
\DeclareMathOperator{\PGL}{PGL}
\DeclareMathOperator{\SL}{SL}
\DeclareMathOperator{\Sp}{Sp}

\DeclareMathOperator{\SO}{SO}

\DeclareMathOperator{\SGO}{SGO}
\DeclareMathOperator{\GO}{GO}
\DeclareMathOperator{\OO}{O}

\DeclareMathOperator{\GL}{GL}
\DeclareMathOperator{\Ad}{Ad}
\DeclareMathOperator{\ad}{ad}
\DeclareMathOperator{\Spin}{Spin}
\DeclareMathOperator{\End}{End}

\newcommand{\F}{\mathbb{F}}
\newcommand{\A}{\mathbb{A}}

\newcommand{\g}{\mathfrak{g}}
\newcommand{\gt}{\tilde{\g}}
\newcommand{\gb}{\bar{\g}}

\newcommand{\s}{\mathfrak{s}}

\newcommand{\hst}{\tilde{\alpha}}

\newcommand{\n}{\mathfrak{n}}
\newcommand{\gl}{\mathfrak{gl}}
\renewcommand{\sl}{\mathfrak{sl}}
\newcommand{\pgl}{\mathfrak{pgl}}
\renewcommand{\sp}{\mathfrak{sp}}
\newcommand{\spin}{\mathfrak{spin}}
\newcommand{\so}{\mathfrak{so}}
\newcommand{\oo}{\mathfrak{o}}

\newcommand{\go}{\mathfrak{go}}

\newcommand{\tor}{\mathfrak{t}}

\DeclareMathOperator{\Sym}{S}

\DeclareMathOperator{\Spec}{Spec}

\newcommand{\T}{\top}

\newcommand{\p}{\mathfrak{p}}
\newcommand{\z}{\mathfrak{z}}

\newcommand{\Gm}{\mathbb{G}_{{\mathrm{m}}}}
\newcommand{\Ga}{\mathbb{G}_{{\mathrm{a}}}}

\newcommand{\drho}{\mathrm{d}\rho}
\newcommand{\dpi}{\mathrm{d}\pi}
\newcommand{\dq}{\mathrm{d}q}

\newcommand{\at}{\tilde{\alpha}}

\newcommand{\la}{\lambda}

\newcommand{\eps}{\varepsilon}

\newcommand{\stbtmat}[4]{\left( \begin{smallmatrix} #1&#2 \\ #3&#4 \end{smallmatrix} \right)}

\newcommand{\Gt}{\widetilde{G}}
\newcommand{\Tt}{\widetilde{T}}
\newcommand{\Gb}{\bar{G}}

\DeclareMathOperator{\Id}{Id}

\usepackage{hyperref}
\hypersetup{
  colorlinks=true, 
   linkcolor=red,  
}
\usepackage[hyphenbreaks]{breakurl}

\numberwithin{equation}{section}

\newcommand{\kx}{k^\times}

\begin{document}

\title[Generically free representations: large representations]{Generically free representations {I}:\\ large representations}

\author[S. Garibaldi]{Skip Garibaldi}
\author[R.M. Guralnick]{Robert M. Guralnick}

\subjclass[2010]{20G05 (primary); 17B10 (secondary)}

\begin{abstract}
For a simple linear algebraic group $G$ acting faithfully on a vector space $V$ and under mild assumptions, we show:
if $V$ is large enough, then the Lie algebra of $G$ acts generically freely on $V$.  That is, the stabilizer in $\Lie(G)$ of a generic vector in $V$ is zero.  The  bound on $\dim V$ grows like $(\rank G)^2$ and holds with only mild hypotheses on the characteristic of the underlying field.  The proof relies on results on generation of Lie algebras by conjugates of an element that may be of independent interest.  We use the bound in subsequent works to determine which irreducible faithful representations are generically free, with no hypothesis on the characteristic of the field.  This in turn has applications to the question of which representations have a stabilizer in general position as well as the determination of the invariants of the representation.  
\end{abstract}

\maketitle
\setcounter{tocdepth}{1}
\tableofcontents

Let $G$ be a simple linear algebraic group over a field $k$ acting faithfully on a vector space $V$.  In the special case $k = \C$, there is a striking dichotomy between the properties of irreducible representations $V$ whose dimension is small (say, $\le \dim G$) versus those whose dimension is large, see \cite{AVE}, \cite{Elashvili:can}, \cite{APopov}, etc.\ for original results and \cite[\S8.7]{PoV} for a survey and bibliography.  For example, if $\dim V < \dim G$, then trivially the stabilizer $G_v$ of a vector $v \in V$ is not 1.  On the other hand (and nontrivially), for $\dim V$ hardly bigger than $\dim G$, the stabilizer $G_v(k)$ for generic $v \in V$ is 1; in this case one says that $V$ is \emph{generically free} or $G$ acts \emph{generically freely} on $V$.
This property has taken on increased importance recently due to applications in Galois cohomology and essential dimension, see \cite{Rei:ICM} and \cite{M:ed} for the theory and \cite{BRV}, \cite{GG:spin}, \cite{Karp:ICM}, \cite{LMMR13}, \cite{Loetscher:fiber}, etc.~for specific applications.

With applications in mind, it is desirable to extend the results on generically free representations to all fields.  The paper
\cite{GurLawther} showed that, for $k$ algebraically closed of any characteristic and $V$ irreducible, $\dim V > \dim G$ if and only if the stabilizer $G_v(k)$ of a generic $v \in V$ is finite.  (This was previously known when $\car k = 0$ \cite{AVE}.)  Moreover, except for the cases in Table \ref{ngenfree}, when $G_v(k)$ is finite it is 1, i.e., the group scheme $G_v$ is infinitesimal.  For applications, it is helpful to know if $G_v$ is not just infinitesimal but is the \emph{trivial} group scheme.  In this paper, we prove the following general bound:

\begin{specthm} \label{MT}
Let $G$ be a simple linear algebraic group over a field $k$ such that $\car k$ is not special for $G$.  If $\rho \!: G \to \GL(V)$ is a representation of $G$ such that $V$ has a $G$-subquotient $X$ with $X^{[\g,\g]} = 0$ and 
$\dim X > b(G)$ for $b(G)$ as in Table \ref{classical.table}, then for generic $v \in V$,  $\Lie(G)_v = \ker \drho$.
\end{specthm}

\begin{table}[t]
\begin{tabular}{cc|c|c||cc|r} 
\multicolumn{4}{c||}{$G$ classical}&\multicolumn{3}{c}{$G$ exceptional}\\
type of $G$&$\car k$&$b(G)$&Reference&type of $G$&$\car k$&$b(G)$ \\ \hline
$A_\ell$&$\ne 2$& $2.25(\ell+1)^2$& Cor.~\ref{A.odd}&$G_2$&$\ne 3$&48 \\
$A_\ell$&$= 2$&$2\ell^2 + 4\ell$ & Cor.~\ref{A.2}&$F_4$&$\ne 2$&240\\
$B_\ell$ ($\ell \ge 3$)&$\ne 2$&$8\ell^2$&Cor.~\ref{BD.odd}&$E_6$&any&360 \\
$C_\ell$ ($\ell \ge 2$)&$\ne 2$&$6\ell^2$&Cor.~\ref{C.odd}&$E_7$&any&630 \\
$D_\ell$ ($\ell \ge 4$)&$\ne 2$&$2(2\ell-1)^2$&Cor.~\ref{BD.odd}&$E_8$&any&1200 \\
$D_\ell$ ($\ell \ge 4$)&$=2$&$4\ell^2$&Cor.~\ref{D.2}
\end{tabular}
\caption{Bound $b(G)$ appearing in Theorem \ref{MT}.  The reference for the exceptional types is Prop.~\ref{ex}.} \label{classical.table}
\end{table}

Of course, $\Lie(G)_v \supseteq \ker \drho$, so equality means that $\Lie(G)_v$ is as small as possible.  In this case, we write that $\Lie(G)$ acts \emph{virtually freely} on $V$.  This notion is the natural generalization of ``generically freely'' to allow for the possibility that $G$ does not act faithfully.  We actually prove a somewhat stronger statement than Theorem \ref{MT}, see Theorem \ref{MTp} below.

Note that $\ker \drho$ can be read off the weights of $V$.  If $\ker \drho$ is a proper ideal in $\Lie(G)$, then (as $\car k$ is assumed not special) it is contained in the center of $\Lie(G)$, i.e., $\Lie(Z(G))$.  The restrictions of $\rho$ to $Z(G)$ and of $\drho$ to $\Lie(Z(G))$ are determined by the equivalence classes of the weights of $V$ modulo the root lattice.

If we restrict our focus to representations $V$ that are restricted and irreducible, Theorem \ref{MT} quickly settles whether $V$ is virtually free for all but finitely many types of $G$:

\begin{speccor} \label{irred}
Suppose $G$ has type $A_\ell$ for some $\ell > 15$; type $B_\ell$, $C_\ell$, or $D_\ell$ with $\ell > 11$; or exceptional type, over an algebraically closed field $k$ such that $\car k$ is not special for $G$.  For $\rho \!: G \to \GL(V)$ an irreducible representation of $G$ whose highest weight is restricted, $\Lie(G)_v = \ker \drho$ if and only if $\dim V > \dim G$.
\end{speccor}

This is proved in Section \ref{irred.sec}.

Note that the bound $b(G)$ from Theorem \ref{MT} holds for most $k$ and is $\Theta(\dim G) = \Theta((\rank G)^2)$ in big-$O$ notation, meaning that it grows like $(\rank G)^2$.  In the special case $\car k = 0$ one can find a similar result in \cite{AndreevPopov} where the bound is $\Theta((\rank G)^3)$, which was used in the (existing) proof of the characteristic 0 version of the results of Section \ref{advert}.
The fact that the exponent in our result is 2 (and not 3) is leveraged in two ways: (1) the restricted irreducible representations not covered by Theorem \ref{MT} and Corollary \ref{irred} are among those enumerated in \cite{luebeck} and (2) it encompasses all but a very small number of tensor decomposable irreducible representations.  We settle these cases in a separate paper, \cite{GG:irred}, because the arguments are rather different and more computational.  Fields with $\car k$ special are treated in \cite{GG:special}, which also includes an example to show that the conclusion of Theorem \ref{MT} does not hold for such $k$.  Combining the results of these two papers with \cite{GurLawther}, we get descriptions of the stabilizer $G_v$ as a group scheme when $V$ is irreducible, which we announce in Section \ref{advert}.  This paper contains the main part of the proof of the results in Section \ref{advert} for Lie algebras.

\subsection*{Remarks on the proof}
Corollary \ref{irred} may be compared to the main result of Guerreiro's thesis \cite{Guerreiro}, which classifies the irreducible $G$-modules that are also $\Lie(G)$-irreducible such that the kernel of $\drho$ is contained in the center of $\Lie(G)$ with somewhat weaker bounds on $\dim V$.  (See also \cite{Auld} and \cite{GG:spin} for other results on specific representations.)  Our methods are different in the sense that Guerreiro relied on computations with the weights of $V$, whereas we largely work with the natural module.  We do refer to Guerreiro's thesis in the proof of Corollary \ref{irred} to handle a few specific representations.

The change in perspective that leads to our stronger results in fewer pages is the replacement of the popular inequality \eqref{ineq.mother}, which involves the action on the specific representation $V$, with \eqref{ineq.gen}, which only involves the dimension of $V$ and properties of the adjoint representation $\Lie(G)$.   Thus our proof of Theorem \ref{MT} depends in only a small way on $V$, providing a dramatic simplification.   Furthermore we prove new bounds on the number of conjugates $e(x)$ of a given
non-central element $x \in \Lie(G)$ that suffice to generate a Lie subalgebra containing the
derived subalgebra (with special care being needed in small characteristic, see, for example, Theorem \ref{qr.thm}); these results should be of independent interest.  Our bounds depend upon the conjugacy class and
give upper bounds for the dimension of fixed spaces for elements in the class.  As a special case, we extend the main result of \cite{CSUW}, see Proposition \ref{gen.thm}.
We note that some generation bounds are 
  known in the setting of groups, see for example
\cite{GurSaxl} or \cite{GordeevSaxl}.

We also prove a result that is of independent interest.  We show in Theorem 
\ref{qr.thm} that the only proper
irreducible Lie subalgebras of $\sl_n$ containing a maximal toral subalgebra
occur in characteristic $2$ and any such is conjugate to the Lie algebra
of symmetric matrices of trace $0$.

\subsection*{Notation}
For convenience of exposition, we will assume in most of the rest of the paper that $k$ is algebraically closed of characteristic $p \ne 0$.  This is only for convenience, as our results for $p$ prime immediately imply the corresponding results for characteristic zero: simply lift the representation from characteristic 0 to $\Z$ and reduce modulo a sufficiently large prime.

We say that $\car k$ is \emph{special} for $G$ if $\car k = p \ne 0$ and the Dynkin diagram of $G$ has a $p$-valent bond, i.e., if $\car k = 2$ and $G$ has type $B_n$ or $C_n$ for $n \ge 2$ or type $F_4$, or if $\car k = 3$ and $G$ has type $G_2$.  (Equivalently, these are the cases where $G$ has a very special isogeny.)  This definition is as in \cite[\S10]{St:rep}, \cite[p.~15]{seitzmem}, and \cite{Premet:supp}; in an alternative history, these primes might have been called ``extremely bad'' because they are a subset of the very bad primes --- the lone difference is that for $G$ of type $G_2$, the prime 2 is very bad but not special.

Let $G$ be an affine group scheme of finite type over $k$.  If $G$ is additionally smooth, then we say that $G$ is an \emph{algebraic group}.  An algebraic group $G$ is \emph{simple} if it is not abelian, is connected, and has no connected normal subgroups $\ne 1, G$; for example $\SL_n$ is simple for every $n \ge 2$.

If $G$ acts on a variety $X$, the stabilizer $G_x$ of an element $x \in X(k)$ is a sub-group-scheme of $G$ with $R$-points
\[
G_x(R) = \{ g \in G(R) \mid gx = x \}
\]
for every $k$-algebra $R$.  A statement ``for generic $x$'' means that there is a dense open subset $U$ of $X$ such that the property holds for all $x \in U$.

If $\Lie(G) = 0$ then $G$ is finite and \'etale.  If additionally $G(k) = 1$, then $G$ is the trivial group scheme $\Spec k$.  (Note, however, that when $k$ has characteristic $p \ne 0$, the sub-group-scheme $\mu_p$ of $\mu_{p^2}$ has the same Lie algebra and $k$-points.  So it is not generally possible to distinguish closed-sub-group schemes by comparing their $k$-points and Lie algebras.)

We write $\g$ for $\Lie(G)$ and similarly $\spin_n$ for $\Lie(\Spin_n)$, etc.  We put $\z(\g)$ for the center of $\g$; it is the Lie algebra of the (scheme-theoretic) center of $G$.  As $\car k = p$, the Frobenius automorphism of $k$ induces a ``$p$-mapping'' $x \mapsto x^{[p]}$ on $\g$.
When $G$ is a sub-group-scheme of $\GL_n$ and $x \in \g$, the element $x^{[p]}$ is the $p$-th power of $x$ with respect to the typical, associative multiplication for $n$-by-$n$ matrices, see \cite[\S{II.7}, p.~274]{DG}.   An element $x \in \g$ is \emph{nilpotent} if $x^{[p]^n} = 0$ for some $n > 0$, \emph{toral} if $x^{[p]} = x$, and \emph{semisimple} if $x$ is contained in the Lie $p$-subalgebra of $\g$ generated by $x^{[p]}$, i.e., is in the subspace spanned by $x^{[p]}, x^{[p]^2}, \ldots$, cf.~\cite[\S2.3]{StradeF}.  
  
\subsection*{Acknowledgements} We thank Brian Conrad for his helpful advice on group schemes.
Part of this research was performed while Garibaldi was at the Institute for Pure and Applied Mathematics (IPAM) at UCLA, which is supported by the National Science Foundation.
Guralnick was partially supported by NSF grants DMS-1265297, DMS-1302886, and DMS-1600056.

\section{Key inequalities} \label{key.sec}

\subsection*{Inequalities}
Put $\lie := \Lie(G)$ and choose a representation $\rho \!: G \to \GL(V)$.  For $x \in \lie$, put
\[
V^x := \{ v \in V \mid \drho(x)v = 0 \}
\]
and $x^G$ for the $G$-conjugacy class $\Ad(G)x$ of $x$.

\begin{lem} \label{ineq} 
For $x \in \g$, 
\begin{equation} \label{mother0}
x^G \cap \g_v = \emptyset \quad \text{for generic $v \in V$}
\end{equation}
is implied by:
\begin{equation} \label{ineq.mother}
  \dim x^G + \dim V^x < \dim V,
\end{equation}
which is implied by:
\begin{equation} \label{ineq.gen}
\parbox{4.25in}{{There exist $e > 0$ and $x_1, \ldots x_e \in x^G$ such that 
the subalgebra $\mathfrak{s}$ of $\g$ generated by $x_1, \ldots, x_e$ has $V^{\mathfrak{s}} = 0$ and
$e\cdot \dim x^G < \dim V$.}}
\end{equation}
\end{lem}

In many uses of \eqref{ineq.gen}, one takes $\mathfrak{s}$ to be $\g$ or $[ \g, \g]$.

\begin{proof} 
Suppose \eqref{ineq.mother} holds and let $v \in V$.   Put 
\[
V(x) := \{ v \in V \mid \text{there is $g \in G(k)$ s.t.\ $xgv = 0$} \} = \bigcup_{y \in x^G} V^y.
\]
Define $\alpha\!: G \times V^x \to V$ by $\alpha(g,w) = gw$, so the image of $\alpha$ is precisely $V(x)$.  The fiber over $gw$ contains $(gc^{-1}, cw)$ for $\Ad(c)$ fixing $x$, and so $\dim V(x) \le \dim x^G + \dim V^x$.  Then \eqref{ineq.mother} implies $\overline{V(x)}$ is a proper subvariety of $V$, whence \eqref{mother0}.
(This observation is essentially in \cite[Lemma 4]{AndreevPopov}, \cite[\S3.3]{Guerreiro}, or \cite[Lemma 2.6]{GG:spin}, for example, but we have repackaged it here for the convenience of the reader.)

Now assume \eqref{ineq.gen}.  
Iterating the formula $\dim(U \cap U') \ge \dim U + \dim U' - \dim V$ for subspaces $U, U'$ of $V$ gives
\begin{equation} \label{ineq.1}
\dim\left(\bigcap\nolimits_i V^{x_i} \right) \ge (\sum\nolimits_i \dim V^{x_i}) - (e-1) \dim V.
\end{equation}
As $\drho$ is $G$-equivariant (and not just a representation of $\g$), we have $\dim V^{x_i} = \dim V^x$.
The left side of \eqref{ineq.1} is zero by  hypothesis, hence $\dim V - \dim V^x \ge \frac1{e} \dim V$ and it follows that $\dim V^x \le (1 - 1/e) \dim V$.  Now $\dim x^G < \frac1{e} \dim V$ implies \eqref{ineq.mother}.
\end{proof}

We will verify \eqref{ineq.mother} in many cases, compare Theorem \ref{MTp}.  To do so, we actually prove \eqref{ineq.gen}, where the inequality only involves $V$ through the term $\dim V$.   This allows us to focus on the element $x$ and its action on the natural module rather than attempting to analyze $V^x$ directly, for which it is natural to require some hypothesis on the structure of $V$ beyond simply a bound on the dimension, such as that $V$ is irreducible as is assumed in \cite{Guerreiro}.  When verifying \eqref{ineq.gen}, one finds that, roughly speaking, when $\dim x^G$ is small, $e$ is large and vice versa.  Therefore, at least for the classical groups, we take some care to bound the product $e \cdot \dim x^G$ instead of bounding each term independently.

\subsection*{Comparing subalgebras}

Exploiting the fact that there are only finitely many $G$-conjugacy classes of toral and nilpotent elements of $\g$ for $G$ semisimple, we obtain as in \cite[\S1]{GG:spin}:
\begin{lem} \label{ineq.spin}
Suppose $G$ is semisimple over an algebraically closed field $k$ of characteristic $p > 0$, and  let $\lsub$ be a $G$-invariant subspace of $\lie$.  
\begin{enumerate}
\item \label{mother.lie2} If inequality \eqref{mother0} holds for every toral or nilpotent $x \in \g \setminus \lsub$, then  $\g_v \subseteq \lsub$ for generic $v \in V$.

\item \label{mother.lie} If $\lsub$ consists of semisimple elements and \eqref{mother0} holds for every $x \in \g \setminus \lsub$ with $x^{[p]} \in \{ 0, x \}$, then $\g_v \subseteq \lsub$  for generic $v$ in $V$. $\hfill\qed$
\end{enumerate}
\end{lem}

Often we apply the preceding lemma with $\lsub = \z(\g)$, the Lie algebra of the center $Z(G)$.  For $G$ reductive, $Z(G)$ is a diagonalizable group scheme \cite[XXII.4.1.6]{SGA3.3:new}, so $Z(G)_v = \ker(\rho \vert_{Z(G)})$.  We immediately obtain:

\begin{lem} \label{central}
Suppose $G$ is reductive.
If, for generic $v \in V$, $\g_v \subseteq \z(\g)$, then $\g$ acts virtually freely on $V$. $\hfill\qed$
\end{lem}

\subsection*{Examples}

\begin{eg}[$\SL_2$] \label{A1.eg}
Recall that an irreducible representation $\rho \!: \SL_2 \to \GL(V)$ of $\SL_2$ is specified by its highest weight $w$, a nonnegative integer.  Let $\car k =: p \ne 0$.  We claim:
\emph{
\begin{enumerate}
	\renewcommand{\theenumi}{\roman{enumi}}
	\item \label{A1.0} If $\car k$ divides $w$ (e.g., if $w = 0$), then $\drho(\sl_2) = 0$.
	\item \label{A1.2} If (a) $w = 1$ or (b) $\car k \ne 2$ and $w = 2$, then $\sl_2$ does not act virtually freely on $V$.
	\item \label{A1.4} If $w = p^e + 1$ for some $e > 0$, then $\sl_2$ acts virtually freely on $V$ but \eqref{ineq.mother} fails for some noncentral $x \in \sl_2$ with $x^{[p]} \in \{ 0, x \}$.
	\item \label{A1.general} Otherwise, \eqref{ineq.mother} holds for noncentral $x \in \sl_2$ with $x^{[p]} \in \{ 0, x \}$, and in particular $\sl_2$ acts virtually freely on $V$.
\end{enumerate}}

To see this, 
write $w = \sum_{i \ge 0} w_i p^i$ where $0 \le w_i < p$.  By Steinberg, $V$ is isomorphic (as an $\SL_2$-module) to $\otimes_i L(\omega_i)^{[p]^i}$, where the exponent $[p]^i$ denotes the $i$-th Frobenius twist, and the irreducible module $L(w_i)$ with highest weight $w_i$ is also the Weyl module with highest weight $w_i$ by \cite{WinterSL2}, of dimension $w_i + 1$.
Thus, as a representation of $\sl_2$, $V$ is isomorphic to a direct sum of $c := \prod_{i>0} (w_i+1)$ copies of $L(w_0)$.  This proves \eqref{A1.0}, so we suppose for the remainder of the proof that $w_0 > 0$.

As in the previous paragraph, $L(1)$ is the natural representation (with generic stabilizer a maximal nilpotent subalgebra) and $L(2)$ (when $p \ne 2$) is the adjoint action on $\sl_2$ (with generic stabilizer a Cartan subalgebra).  This verifies \eqref{A1.2}.

We investigate now \eqref{ineq.mother}.  For $x$ nonzero nilpotent or noncentral toral, we have $\dim(x^{\SL_2}) = 2$.  For $x$ nonzero nilpotent, $L(w_0)^x$ is the highest weight line.  If $x^{[p]} = x$, then up to conjugacy $x$ is diagonal with entries $(a, -a)$ for some $a \in \F_p$; as $x$ is non-central, $p \ne 2$ and $\dim L(w_0)^x = 0$ or 1 depending on whether $w_0$ is odd or even.  Assembling these, we find
$\dim (x^{\SL_2}) + \dim L(w)^x \le 2+c$ with equality for $x$ nonzero nilpotent, whereas $\dim L(w) = cw_0 + c$.  We divide the remaining cases via the product $cw_0$, where we have already treated the case \eqref{A1.2} where $c = 1$ and $w_0 = 1$ or 2. 

Suppose  $c = 2$ and $w_0 = 1$, so we are in case \eqref{A1.4}.  The action of $\sl_2$ on $V$ via $\drho$ is the same as the action of $\sl_2$ on two copies of the natural module, equivalently, on 2-by-2 matrices by left multiplication.  A generic matrix $v$ is invertible, so $(\sl_2)_v = 0$.  Yet we have verified in the previous paragraph that \eqref{ineq.mother} fails for $x$ nonzero nilpotent, proving \eqref{A1.4}.

The case \eqref{A1.general} is where $cw_0 > 2$, where we have verified \eqref{ineq.mother}, completing the proof of the claim.

As a corollary, we find: \emph{$\sl_2$ fails to act virtually freely on $V$ if and only if (a) $w = 1$ or (b) $\car k \ne 2$ and $w = 2$.}  Moreover, when $\sl_2$ acts faithfully on $V$ (i.e, $w_0$ is odd), we have: \emph{$\sl_2$ fails to act generically freely on $V$ if and only if $w = 1$ if and only if $\dim V \le \dim \SL_2$.}
\end{eg}

\begin{eg} \label{dual.eg} 
Let $x \in \g$.  \emph{If $\dim x^G + \dim (V^*)^x < \dim (V^*)$, then \eqref{ineq.mother} holds for $x$.}  This is obvious, because $\drho(x)$ and $-\drho(x)^\top$ have the same rank.
\end{eg}

\section{Interlude: semisimplification} 

For Theorem \ref{MT}, we consider representations $V$ of $G$ that need not be semisimple.  For each chain of submodules $0 =: V_0 \subseteq V_1 \subseteq V_2 \subseteq \cdots \subseteq V_n := V$ of $G$, we can construct the $G$-module $V' := \oplus_{i=1}^n V_i / V_{i-1}$.  For example, if each $V_i / V_{i-1}$ is an irreducible (a.k.a.~simple) $G$-module then $V'$ is the \emph{semisimplification} of $V$.  In this section, we discuss to what extent results for $V$ correspond to results for $V_i/V_{i-1}$ and for $V'$, using the notation of this paragraph and writing $\rho \!: G \to \GL(V)$ and $\rho' \!: G \to \GL(V')$ for the actions.  

\subsection*{From the subquotient to $V$}

\begin{eg} \label{ineq.sub}
Suppose that for some $x \in \g$ and some $1 \le i \le n$, we have
\[
\dim x^G + \dim (V_i/V_{i-1})^x < \dim (V_i/V_{i-1}).
\]
\emph{We claim that \eqref{ineq.mother} holds for $x$.}  By induction it suffices to consider the case $i = 2$ and a chain $V_1 \subseteq V_2 \subseteq V$.

Suppose first that $V_1 = 0$.  Then
$\dim x^G + \dim V^x \le \dim x^G + \dim V_2^x + \dim V/V_2$, whence the claim. Now suppose that $V_2 = V$, so $(V_2/V_1)^*$ is a submodule of $V^*$; the claim follows by Example \ref{dual.eg}.  Combining these two cases gives the full claim.
\end{eg}

There is an analogous statement about the dimension of generic stabilizers.

\begin{eg} \label{summand}
For each $1 \le i \le n$ and generic $v \in V$ and generic $w \in V_i/V_{i-1}$, we claim that $\dim \g_v \le \dim \g_w$.  Take $w \in V_i / V_{i-1}$ to be the image of a generic $\hat{w} \in V_i$.  Then $\dim \g_v \le \dim \g_{\hat{w}}$ by upper semicontinuity of dimension and clearly $\dim \g_{\hat{w}} \le \dim \g_w$.
\end{eg}

\subsection*{From $V'$ to $V$}

\begin{eg} \label{ineq.ss}
When checking the inequality \eqref{ineq.mother}, it suffices to do it for $V'$.  More precisely, for $x \in \g$, we have: \emph{If $\dim x^G + \dim (V')^x < \dim V'$, then $\dim x^G + \dim V^x < \dim V$.}  This is obvious because $\dim V^x \le \sum \dim (V_i/V_{i-1})^x$.
\end{eg}

The following strengthens Example \ref{summand}.

\begin{prop}
For generic $v \in V$ and $v' \in V'$, we have $\dim \g_v \le \dim \g_{v'}$.  
\end{prop}

\begin{proof}
By induction on the number $n$ of summands in $V'$, we may assume that $V' = W \oplus V/W$ for some $\g$-submodule $W$ of $V$.

Suppose first that $\dim V/W = 1$.  Pick $v \in V$ with nonzero image $\bar{v} \in V/W$.  Put $\tor := \{ x \in \g \mid \drho(x)v \in W \}$, a subalgebra of $\g$ sometimes called the transporter of $v$ in $W$.  
A generic vector $v' \in V'$ is of the form $w \oplus c \bar{v}$ for $w \in W$ and $c \in \kx$.  Evidently, $\g_{v'} = \tor_w$.  By upper semicontinuity of dimension, $\dim \tor_{v_0} \le \dim \tor_w$ for generic $v_0 \in V$.  On the other hand, writing $v_0 = w_0 + \la v$ for $\la \in \kx$ and $w_0 \in W$, for $x \in \g_{v_0}$ we find $\drho(x)v = -\frac1\la \drho(x) w_0 \in W$, so $\g_{v_0} = \tor_{v_0}$, proving the claim.

In the general case, pick a splitting $\phi \!: V/W \hookrightarrow V$ and so identify $V$ with $V'$ as vector spaces.  We may intersect open sets defining generic elements in $V$ and $V'$ and so assume the two notions agree under this identification.  Let $v := w + \phi(\bar{v})$ be a generic vector in $V$, where $w \in W$ and $\bar{v} \in V/W$ is the image of $v$; $v' := w \oplus \bar{v}$ is a generic vector in $V'$.
Defining $\tor$ as in the previous paragraph, we have $\g_v, \g_{v'} \subseteq \tor$. 
Replacing $\g$, $V$, $V'$ with $\tor$, $W + kv$, $W \oplus k\bar{v}$ and referring to the previous paragraph gives the claim.
\end{proof}

If $\g$ acts generically freely on $V'$ (i.e., $\g_{v'} = 0$), then the proposition says that $\g$ acts generically freely on $V$.  This immediately gives the following statement about group schemes:

\begin{cor}
If $G_{v'}$ is finite \'etale for generic $v' \in V'$, then $G_v$ is finite \'etale for generic $v \in V$. $\hfill \qed$
\end{cor}

While generic freeness of $V'$ implies generic freeness of $V$ for the action by the Lie algebra $\g$, it does not do so for the action by the algebraic group $G$, as the following example shows.

\begin{eg}
Take $G = \Ga$ acting on $V = \A^3$ via 
\[
\rho(r) := \left( \begin{smallmatrix} 1&r &r^p \\ 0 & 1 & 0 \\ 0 & 0 & 1 \end{smallmatrix} \right).
\]
Let $V_2 \subset V$ be the subspace of vectors whose bottom entry is zero.  Then $G$ acts on $V_2$ via $r \mapsto \stbtmat{1}{r}{0}{1}$ and in particular a generic $v_2 \in V_2$ has $G_{v_2} = 1$.  On the other hand, a generic vector $v := \left( \begin{smallmatrix} x \\ y \\ z \end{smallmatrix} \right)$ in $V$ has $G_v$ the \'etale subgroup with points $\{ r \mid ry + r^p z = 0 \}$, i.e., the kernel of the homomorphism $zF + y \Id \!: \Ga \to \Ga$ for $F$ the Frobenius map.
\end{eg}

Direct sums have better properties with respect to calculating generic stabilizers, see for example \cite[Prop.~8]{Popov:Borel} and \cite[Lemma 2.15]{Loetscher:edsep}.

\section{Lemmas on the structure of $\g$}

When $\car k$ is not zero (more precisely, not very good), then it may happen that $\g$ depends not just on the isogeny class of $G$, but may depend on $G$ up to isomorphism.  Moreover, $\g$ need not be perfect even when $G$ is simple.  In this section we record for later use some fac ts that do hold in this level of generality.


\begin{lem} \label{quo}
Let $G$ be a simple algebraic group over $k$ such that $(G, \car k) \ne (\Sp_{2n}, 2)$ for all $n \ge 1$.  Put $\pi \!: \Gt \to G$ for the simply connected cover of $G$ and $\gt := \Lie(\Gt)$.  
Then:
\begin{enumerate}
	\item \label{quo.der} $[\g, \g] = \dpi(\gt)$.
	\item \label{quo.kill} If $V$ is an irreducible representation of $G$ whose highest weight is restricted, then $V^{[\g, \g]} = 0$.
\end{enumerate}
\end{lem}

\begin{proof}
The map $\dpi$ restricts to an isomorphism $\gt_\alpha \xrightarrow{\sim} \g_\alpha$ for each root $\alpha$, and in particular $\dpi(\gt) \supseteq \qform{\g_\alpha}$, an ideal in $\g$.  As $\g/\qform{\g_\alpha}$ is abelian, $\qform{\g_\alpha} \supseteq [\g, \g]$. 

Conversely, $[\gt, \gt] = \gt$, see  \cite[Lemma 2.3(ii)]{Premet:supp} if $\car k$ is not special and \cite[6.13]{Hogeweij:th} in general.  So $\dpi(\gt) = \dpi([\gt, \gt]) \subseteq [\g, \g]$.

To see \eqref{quo.kill}, 
write the highest weight $\la$ of $V$ as a sum of fundamental dominant weights $\la = \sum c_i \omega_i$.  As $\la$ is restricted, there is some $c_i \in \Z$ whose image in $k$ is not zero.  Put $\alpha$ for the simple root such that $\qform{\omega_i, \alpha^\vee} = c_i$.  Writing $x_\alpha, x_{-\alpha}$ for basis elements of the root subalgebras for $\pm \alpha$ and $v$ for a highest weight vector in $V$, we have $x_\alpha x_{-\alpha} v = \qform{\la, \alpha^\vee} v \ne 0$ as in the proof of \cite[Lemma 4.3(a)]{St:rep}, so $V^{\dpi(\gt)} = V^{[\g, \g]}$ is a proper submodule of $V$, hence is zero.
\end{proof}

\begin{cor} \label{quo.stab}
Let $G$ be a simple algebraic group over $k$ and put $\pi \!: \Gt \to G$ for the simply connected cover.  If $\rho \!: G \to \GL(V)$ is a representation such that $\drho\, \dpi = 0$ (i.e., $\gt$ acts trivially on $V$), then $\g$ acts virtually freely on $V$.
\end{cor}

\begin{proof} 
If $G$ is simply connected --- i.e., $G = \Gt$ --- then $\drho = 0$ and this is trivial.  So assume $G$ is not simply connected and apply Lemma \ref{quo}.  There is a torus $T$ in $G$ such that $\g = [\g, \g] + \tor$ as a vector space.  In particular, the images of $\g$ and $\tor$ in $\gl(V)$ are the same.  The image consists of simultaneously diagonalizable matrices, so $\tor$ acts virtually freely, ergo the same is true for $\g$.
\end{proof}

\begin{eg}[$\PGL_2$] \label{PGL2}   
Let $\rho \!: \PGL_2 \to \GL(V)$ be an irreducible representation.  The composition $\SL_2 \to \PGL_2 \xrightarrow{\rho} \GL(V)$ is an irreducible representation $L(w)$ of $\SL_2$ as in Example \ref{A1.eg} with $w$ even.  We claim that \emph{$\PGL_2$ fails to act virtually freely on $V$ if and only $\car k \ne 2$ and $w = 2$.}

If $\car k \ne 2$, the induced map $\sl_2 \to \pgl_2$ is an isomorphism and the claim follows from Example \ref{A1.eg}.

If $\car k = 2$, then, as $w$ is even, the representation of $\SL_2$ is isomorphic to the Frobenius twist $L(w/2)^{[2]}$ and $\sl_2$ acts trivially (and $\rho$ is not faithful).  By Corollary \ref{quo.stab}, the action of $\pgl_2$ is virtually free, verifying the claim.

Suppose now additionally that $\rho$ is faithful, whence $\car k \ne 2$.  Applying the above, we find that \emph{$\pgl_2$ fails to act generically freely if and only if $w = 2$ if and only if $\dim V \le \dim \PGL_2$.}
\end{eg}

\begin{eg}[adjoint representation] \label{adjoint}
Let $G$ be a simple algebraic group and put $L(\hst)$ for the irreducible representation with highest weight the highest root $\hst$. It is a composition factor of the adjoint module.  

If $\car k = 2$ and $G$ has type $C_n$ for $n \ge 1$ (including the cases $C_1 = A_1$ and $C_2 = B_2$), then $\g$ acts virtually freely on $L(\hst)$.  In case $G$ is simply connected, $L(\hst)$ is a Frobenius twist of the natural representation of dimension $2n$ (since $\hst$ is divisible by 2 in the weight lattice), so $\g$ acts as zero (and, in particular, virtually freely); compare Example \ref{A1.eg}\eqref{A1.0} for the case $n = 1$.  If $G$ is adjoint, then we apply Corollary \ref{quo.stab}.

Now suppose that $\car k$ is not special for $G$ and $(\type G, \car k) \ne (A_1, 2)$.
Put $\pi \!: \Gt \to G$ for the simply connected cover of $G$.  The hypotheses give that $L(\hst) \cong \gt / \z(\gt)$ as $G$-modules and that Cartan subalgebras of $\gt$ and $\g$ are Lie algebras of maximal tori.
It follows, then, that there is an open subset $U$ of $\gt$ that meets $\Lie(\Tt)$ for every maximal torus $\Tt$ of $\Gt$ such that for $a \in U$ the subalgebra
$\Nil(a, \gt) := \cup_{m > 0} \ker (\ad a)^m$ has minimal dimension (i.e., $a$ is regular in the sense of \cite[\S{XIII.4}]{SGA3.2}).  Pick $a \in U \cap \Tt$, put $\bar{a} \in L(\hst)$ for the image of $a$, and set 
\[
\gt_{\bar{a}} = \{ x \in \gt \mid \ad(x) \bar{a} \in \z(\gt) \}.
\]
Then 
\[
\Lie(\Tt) \subseteq \gt_{\bar{a}} \subseteq \Nil(a, \gt) = \Lie(\Tt),
\]
where the last equality is by \cite[Cor.~XIII.5.4]{SGA3.2}.
The image $T$ of $\Tt$ in $G$ is a maximal torus that fixes $\bar{a}$, so $\g_{\bar{a}}$ is generated by $\Lie(T)$ and the root subgroups it contains.  But any such root subgroup would be the image of the corresponding root subgroup of $\gt$, which does not stabilize $\bar{a}$, and therefore $\g_{\bar{a}} = \Lie(T)$.  In particular, $\g$ does not act virtually freely on $L(\hst)$.
\end{eg}

%

\begin{lem} \label{quo.gen}
Suppose $G$ is a simple algebraic group such that $\car k$ is not special and $(G, \car k) \ne (\SL_2, 2)$.
If $\s$ is a subalgebra of $\g$ such that $\s + \z(\g) \supseteq [\g, \g]$, then $\s \supseteq [\g, \g]$.
\end{lem}

For the excluded case where $G = \SL_2$ and $\car k = 2$, $\z(\g) = [\g, \g]$ is the Lie algebra of every maximal torus.

\begin{proof}
We may assume that $\z(\g) \ne 0$, and in particular the center of $G$ is not \'etale and $G$ does not have type $A_1$.

If $G$ is equal to its simply connected cover $\Gt$,  then for each $g \in G(k)$, there is $z_g \in \z(\g)$ such that $z_g + gx_{\at} \in \s$, where $\at$ denotes the highest root.  Thus, $\s$ contains $[z_g + gx_{\at}, z_{g'} + g'x_{\at}] = [gx_{\at}, g'x_{\at}]$ for all $g, g' \in G(k)$, hence $\s = \g$ by \cite[Lemma 2.3(ii)]{Premet:supp}.

Suppose now that $G \ne \Gt$.  We may replace $\s$ with $\s \cap [\g, \g]$ and so assume $\s \subseteq [ \g, \g]$.  Put $\tilde{\s}$ for the inverse image $\dpi^{-1}(\s)$ of $\s$ in $\gt$ and $q \!: G \to \Gb$ for the natural map to the adjoint group.  The kernels of $\dq \, \dpi$ and $\dpi$ are the centers of $\gt$ and $\g$ respectively, so
 $\dq\, \dpi (\tilde{s}) = \dq(\s) \supseteq \dq([\g, \g])$ by hypothesis, which equals $\dq\, \dpi(\gt)$ by Lemma \ref{quo}\eqref{quo.der}.  We are done by the case where $G$ is simply connected.
\end{proof}

\section{Deforming semisimple elements to nilpotent elements} \label{deforming}

For $x \in \g$, we use the shorthand $x^{\Gm G}$ for the orbit of $x$ under the subgroup of $\GL(\g)$ generated by $\Gm$ and $\Ad(G)$.  
For $y$ in the closure of $x^{\Gm G}$, $\dim V^x \le \dim V^y$ by upper semicontinuity of dimension.

\begin{eg} \label{deform.eg}
Suppose that $x \in \g$ is non-central semisimple and let $\mathfrak{b}$ be a Borel subalgebra containing $x$.   Because $x$ is not central, there is a root subgroup $U_{\alpha}$ in the corresponding Borel subgroup that does not commute with $x$.   This implies that $x + \lambda y$ is in the same $G$-orbit as $x$
for all $\la \in k$ and $y$ in the corresponding root subalgebra, and similarly $\la x + y$ is in the same $G$-orbit as $\la x$ and in particular $y$ is in the closure of $x^{\Gm G}$, so $\dim V^x = \dim V^{\la x + y} \le \dim V^y$.
\end{eg}

\begin{lem} \label{deform}
Suppose $k$ is algebraically closed and let $G = \GL_n$ or $\SL_n$.
Let $x \in \lie$ be a semisimple element.   Then there exists a nilpotent element
$y \in \lie$ such that the following hold:
\begin{enumerate}
\item \label{deform.1} The $\Ad(G)$-orbits of $x$ and $y$ have the same dimension.
\item  $y$ is in the closure of $x^{\Gm G}$.
\item If the matrix $x$ has $r$ distinct eigenvalues, then $y^{r-1} \ne 0$ and $y^r = 0$. 
In particular, if $p := \car k \ne 0$ and $x$ is toral, then $y^{[p]}=0$.
\item \label{deform.rank} The rank of $y$ is the codimension of the largest eigenspace of $x$.  In particular, if $0$ is the eigenvalue of $x$ with greatest multiplicity, then $\rank y = \rank x$.
\item \label{deform.4} If $V$ is a finite dimensional rational $G$-module, then $\dim V^y \ge \dim V^x$
and $\dim y^G + \dim V^y \ge \dim x^G + \dim V^x$.
\end{enumerate}
\end{lem}

\begin{proof}
Suppose first that $G=\GL_n$.  We may assume that $x$ is diagonal. 
Permuting the basis so that vectors with the same eigenvalue are adjacent, we may assume that $x$ has $a_1, \ldots, a_r$  down the diagonal $a_i$ appearing $n_i$ times
and $n_1 \ge n_2 \ge \ldots \ge n_r$.
 The centralizer of $x$ in $\GL_n$ is $\prod_i \GL_{n_i}$ of dimension $\sum n_i^2$.

Let $y$ be the block upper triangular matrix (with the blocks corresponding to the
eigenspaces of $x$) such that the only nonzero blocks are the ones corresponding to
the $a_i, a_{i+1}$ block.  In that block, take $y$ to have $1$'s on the diagonal
and $0$'s elsewhere; this block has rank $n_{i+1}$ so $\rank y = \sum_{i=2}^r n_i$  as claimed in \eqref{deform.rank}.

After conjugation
by a permutation matrix, we  deduce that the size of the Jordan blocks in the Jordan form of $y$ are given by the partition of $n$ conjugate to $(n_1, n_2, \ldots, n_r)$.  The centralizer of such a matrix has dimension $\sum n_i^2$, cf.~\cite[p.~E-84, 1.7(iii); E-85, 1.8]{SpSt} or \cite[p.~14]{Hum:conj} and so \eqref{deform.1} holds.

It follows that the largest Jordan block of $y$ has size $r$ whence the minimal polynomial
of $y$ has degree $r$ (equal to the degree of the minimal polynomial of $x$). 

Clearly, $x + ty \in x^G$ whence $y$ is in the closure of $x^{\Gm G}$ and so
$\dim V^x \le \dim V^y$.  This fact and \eqref{deform.1} imply the last inequality in \eqref{deform.4}.

If $x$ is toral, then $x$ has all eigenvalues in $\F_p$ and so $r \le p$, whence
$y^{[p]}=0$. 
 
 For $G = \SL_n$, each toral element is also toral in $\GL_n$ and one takes $y$ as in the $\GL_n$ case.
\end{proof} 

\subsection*{Generation} 
\newcommand{\Xb}{\overline{X}}

\begin{lem} \label{gen}
Let $\rho \!: G \to \GL(V)$ be a representation of an algebraic group over a field $k$.
Let $X$ be an irreducible and $G$-invariant subset of $\g$ such that $X$ is open in $\Xb$.  If, for some $Y \subseteq \Xb$, there exist $e > 0$ and $y_1, \ldots, y_e \in Y(k)$ that generate a subalgebra of $\g$ 
\begin{enumerate}
\item \label{opengen.trivial}  that has dimension at least $d$, for some $d$;
\item \label{opengen.inv} that leaves no $d$-dimensional subspace of $V$ invariant for some $d$;
\item \label{opengen.1} containing a $G$-invariant subalgebra $M$ of $\g$ such that $M/N$ is an irreducible $M$-module and $\dim \g/M < \dim M/N$, for some $G$-submodule $N$ of $M$; or
\item \label{opengen.2} containing a strongly regular semisimple element (as defined in Example \ref{sr.eg}), \end{enumerate}
then $e$ generic elements of $X$ do so as well.
\end{lem}

We will use this lemma with $X = x^G$ and $Y = y^G$ for $x$ and $y$.  For a description of which nilpotent $y$ lie in $\overline{x^G}$ for a given $x$, we refer to \cite[3.10]{Hesselink:sing} for type $A$ and, when $\car k \ne 2$, types $B$, $C$, and $D$.  (A description can also be found in \cite[\S6.2]{CMcG}.)  For the other cases we use Lemma \ref{root.2}.

Alternatively, one can take $x$ and $y$
as in Example \ref{deform.eg} or Lemma \ref{deform} and set $X = x^{\Gm G}$ and $Y = y^G$. 

\begin{proof}[Proof of Lemma \ref{gen}]
For each of \eqref{opengen.trivial}--\eqref{opengen.2}, we consider the subset $U$ consisting of those $(y_1, \ldots, y_e)$ in a product $\Xb^{\times e}$ of $e$ copies of $\Xb$ that generate a subalgebra satisfying the given condition.  Fix an $e > 0$ so that $U(k)$ is nonempty.  It suffices to observe that $U$ is open in $\Xb$, which is obvious for \eqref{opengen.trivial}.  Case \eqref{opengen.inv} is argued as in \cite[Lemma 3.6]{BGGT}.

For \eqref{opengen.1}, consider the set $U'$ of $(y_1, \ldots, y_e) \in \Xb^{\times e}$ such that $y_1, \ldots y_e$ generate a subalgebra $Q$ with $Q$ acting irreducibly on $M/N$ and $\dim Q \ge \dim M$; it is open as in \eqref{opengen.trivial} and \eqref{opengen.inv}.  We claim that $U' = U$; the containment $\supseteq$ is clear.  Conversely, if $(y_1, \ldots, y_e)$ is in $U' \setminus U$, then $Q \cap M \subseteq N$ and $\dim Q \le \dim \g/M + \dim N < \dim M$, a contradiction.  

For \eqref{opengen.2}, the hypothesis is that some word $w$ in variables is strongly regular semisimple for some collection of $e$ elements of $Y(k)$.   Since being strongly regular semisimple is an open condition, it follows that $w$ is generically strongly regular semisimple.
\end{proof}

We also use the lemma in the form of the following corollary.

\begin{cor} \label{gen.cor}
Let $G$ be a simple algebraic group over a field $k$ such that $\car k$ is not special for $G$ and $(\type G, \car k) \ne (A_1, 2)$.
Let $X$ be an irreducible and $G$-invariant subset of $\g$ such that $X$ is open in $\Xb$.  If, for some $Y \subseteq \Xb$, there exist $e > 0$ and $y_1, \ldots, y_e \in Y(k)$ that generate a subalgebra of $\g$ containing $[\g, \g]$,  then $e$ generic elements of $X$ do so as well.
\end{cor}

\begin{proof}
Set $M := \dpi(\gt) = [\g, \g]$ (Lemma \ref{quo}\eqref{quo.der}) and $N := \dpi(\z(\gt)) = [\g, \g] \cap \z(\g)$.  Then $M/N$ is, as a $G$-module, $L(\at)$, an irreducible representation of $M$ (Example \ref{adjoint}).  Moreover, $\dim \g/M \le \dim \z(\gt) \le 2 < \dim M/N$.  Apply  Lemma \ref{gen}\eqref{opengen.1}.
\end{proof}

\section{Quasi-regular subalgebras}

For this section, let $T$ be a maximal torus in a reductive algebraic group $G$ over an algebraically closed field $k$.  Writing $\tor := \Lie(T)$ and $\g := \Lie(G)$, the action of $T$ on $\g$ gives the Cartan decomposition $\g = \tor \oplus \bigoplus_{\alpha \in \Phi} \g_\alpha$ where $\Phi$ is the set of roots of $G$ with respect to $T$ and $\g_\alpha$ is the 1-dimensional root subalgebra for the root $\alpha$.  (Note that the action by $\tor$ induces a direct sum decomposition on $\g$ that need not be as fine when $\car k = 2$, for in that case $\alpha$ and $-\alpha$ agree on $\tor$, and if furthermore $G = \Sp_{2n}$ for $n \ge 1$, then the centralizer of $\tor$ in $\g$, the Cartan subalgebra, properly contains $\tor$.)  We say that a subalgebra $L$ of $\g$ is \emph{quasi-regular with respect to $T$} if 
\[
L = (L \cap \tor) \oplus \begin{cases}
\oplus_{\alpha \in \Phi} (L \cap \g_\alpha) & \text{if $\car k \ne 2$} \\
\oplus_{\alpha \in \Phi^+} (L \cap \g_{\pm\alpha}) & \text{if $\car k = 2$} 
\end{cases}
\]
as a vector space,
where $\g_{\pm \alpha} := \g_\alpha \oplus \g_{-\alpha}$ and $\Phi^+$ denotes the set of positive roots relative to some fixed ordering.  We say simply that $L$ is quasi-regular if it is quasi-regular with respect to some torus $T$.

For $L$ quasi-regular with respect to $T$, $\tor$ evidently normalizes $L$, i.e., $L + \tor$ is also a quasi-regular subalgebra.

\begin{eg} \label{sr.eg}
Suppose there is a $t \in \tor \cap L$ such that 
\begin{equation} \label{sr}
\pm \alpha(t) \ne \pm \beta(t) \quad \text{for all $\alpha \ne \beta \in \Phi^+ \cup \{ 0 \}$,}
\end{equation}
i.e., that has the same eigenspaces on $\g$ as $\tor$.  (We call such a $t$ \emph{strongly regular}.)
Put $m(x)$ for the minimal polynomial of $\ad(t)$.  For each $\alpha \in \Phi \cup \{ 0 \}$, evaluating $m(x)/(x-\alpha(t))$ at $\ad(t)$ gives a linear map $\g \to \g$ with image $\g_\alpha$ (if $\car k \ne 2$) or $\g_{\pm \alpha}$ (if $\car k = 2$).  Restricting $t$ to $L$ shows that $L \cap \g_\alpha$ or $L \cap \g_{\pm \alpha}$ is contained in $L$, i.e., \emph{$L$ is quasi-regular.}
\end{eg}

\begin{eg} \label{sln.qr}
Suppose $G = \SL_n$ or $\GL_n$ for $n \ge 4$.  If $L$ contains a copy of $\sl_{n-1}$ (say, the matrices with zeros along the rightmost column and bottom row), then $L$ is quasi-regular.  Indeed, taking $T$ to be the diagonal matrices in $G$ and $t \in \tor$ to have distinct indeterminates in the first $n-2$ diagonal entries and a zero in the last diagonal entry, we find that $t$ satisfies \eqref{sr}.  This $L$ is quasi-regular, but need not be regular, in the sense that it need not contain a maximal toral subalgebra of $\g$.
\end{eg}

\begin{rmk} \label{maxrank}
Suppose $\car k \ne 2$ and $\g = \gl_n$, $\sl_n$, $\so_n$, or $\sp_{2n}$.  \emph{If $\lsub$ is a Lie subalgebra that contains a maximal toral subalgebra $\tor$} (so $\lsub$ is quasi-regular) \emph{and acts irreducibly on the natural module, then $\lsub = \g$.}  To see this, note that $\lsub$ is a sum of $\tor$ and the root spaces it contains (using the $\car k \ne 2$), and so is determined by $\tor$ and a closed subset of the root system of $\g$, whose classification over $k$ is the same as the Borel-de Siebenthal classification over $\C$.  

The claim is clear if $\lsub$ is contained in a maximal parabolic subalgebra, for such subalgebras act reducibly (even stabilizing a totally singular subspace for $\g = \so_n$ or $\sp_{2n}$), see for example  \cite[\S3]{CG}.  Otherwise, $\lsub$ stabilizes a nondegenerate subspace (compare for example \cite[Table 9]{Dynk:ssub}) and again the claim follows.  See \cite[Lemma 3.6]{BGGT} for a similar statement on the level of groups.
\end{rmk}

\subsection*{The subsystem subalgebra} Suppose $L$ is a quasi-regular subalgebra of $\g$ with respect to $T$.  Define $L_0$ to be the subalgebra of $L$ generated by the $L \cap \g_\alpha$ for $\alpha \in \Phi$.

\begin{lem} \label{L0.ideal}
If
\begin{enumerate}
\item $\car k \ne 2$ or
\item \label{L0.ideal.2} $\car k = 2$, $\Phi$ is irreducible, and all roots have the same length,
\end{enumerate}
then $L_0$ is an ideal in $L + \tor$.
\end{lem}

\begin{proof}
If $\car k \ne 2$, then $L_0 \cap \g_\alpha = L \cap \g_\alpha$ for all $\alpha$ and the claim is trivial, so assume \eqref{L0.ideal.2} holds.
As $L_0$ is evidently stable under $\ad \tor$, it suffices to check that, for $x_\beta \in \g_\beta$, $x_{-\beta} \in \g_{-\beta}$, $c \in k$ such that $x_\beta + cx_{-\beta} \in L$, and $x_\alpha \in L \cap \g_\alpha \subseteq L_0$ that $L_0$ contains
\[
[x_\beta + c x_{-\beta}, x_\alpha] = [x_\beta, x_\alpha] + c[x_{-\beta}, x_\alpha] .
\]
However, by hypothesis $\alpha + \beta$ and $\alpha - \beta$ cannot both be roots, so at least one of the two terms in the displayed sum is zero and the expression belongs to $L \cap \g_{\alpha + \beta}$ or $L \cap \g_{\alpha - \beta}$, hence to $L_0$.
\end{proof}

\begin{eg}
Let $L$ be the space of
symmetric $n$-by-$n$ matrices in $\gl_n$.  It is a Lie subalgebra when $\car k = 2$, and, in that case, it is quasi-regular with respect to the maximal torus $T$ of diagonal matrices in $\GL_n$ and $L_0 = 0$.
\end{eg}

\begin{lem} \label{L0.subsystem}
Suppose $L$ is a quasi-regular subalgebra of $\gl(V)$ with respect to a maximal torus $T$.  Then $L_0$ is irreducible on $V$ if and only if $L_0 + \tor$ is irreducible on $V$ if and only if $L_0 = \sl(V)$, if and only if $L_0 + \tor = \gl(V)$.
\end{lem}

\begin{proof}
The algebra $L_0$ is $(L_0 \cap \tor) \oplus \bigoplus_{\alpha \in S} \g_\alpha$ where $S$ is a closed subsystem of a root system of type $A$.  Therefore $S = \Phi$ (in which case $L_0$ acts irreducibly and $L_0 = \sl(V)$) or $S$ is contained in a proper subsystem (which normalizes a proper $T$-invariant subspace of $V$).
\end{proof}

\subsection*{Application to type $A$}

\begin{thm} \label{qr.thm}
Suppose $L$ is a subalgebra of $\gl_n$ for some $n \ge 2$ that is quasi-regular and acts irreducibly on the natural representation of $\gl_n$.  Then 
\begin{enumerate}
\item \label{qr.sl} $L$ contains $\sl_n$, or
\item \label{qr.symm} $\car k = 2$ and $L$ is $\GL_n$-conjugate to a subalgebra of symmetric $n$-by-$n$ matrices containing the alternating matrices.
\end{enumerate}
\end{thm}

\begin{proof}
Let $T$ be the maximal torus with respect to which $L$ is quasi-regular.  After conjugation by an element of $\GL_n(k)$, we may assume that $T$ is the diagonal matrices.  If $L_0 = L$ or even $L_0 + \tor = \gl_n$, Lemma \ref{L0.subsystem} gives that $L$ contains $\sl_n$.

\subsubsection*{Case: $L_0 \ne 0$} Suppose $L_0 \ne 0$.  We claim that \eqref{qr.sl} holds.  Replacing $L$ with $L + \tor$, we may assume that $\tor \subseteq L$.  We claim that $L_1 := L_0 + \tor$ acts irreducibly on the natural representation $V := k^n$ of $\gl_n$. If $\dim V =2$, the result is clear.  So we assume that $\dim V \ge 3$.

Suppose $W \subseteq V$ is a subspace on which $L_1$ acts nontrivially and irreducibly.  Conjugating by a monomial matrix, we may assume that $W$ is the subspace consisting of vectors whose nonzero entries are in the first $w := \dim W$ coordinates.  If $w =1$,  we can apply the graph automorphism that inverts $T$ and permutes the root spaces and get a possibly different
subalgebra $L'$ which leaves invariant a hyperplane.   Of course, it suffices to prove the result for
$L'$ and so we may take $w \ge 2$.  Now $L_1 \cap \gl(W)$ is a quasi-regular subalgebra of $\gl(W)$ acting irreducibly on $W$ and it is generated by $\tor \cap \gl(W)$ and those $\g_\alpha$ contained in $L$, so by Lemma \ref{L0.subsystem} it equals $\gl(W)$.

If $W \ne V$, then there is a $\beta \in \Phi$ such that $\g_{\pm \beta} \cap L_0 = 0$ yet $(\g_{\pm \beta} \cap L) W \not\subseteq W$.  That is, there exists $i > w$ and $j \le w$ such that $E_{ij} - cE_{ji} \in L$ for some $c \in k^\times$, where $E_{ij}$ denotes the matrix whose unique nonzero entry is a 1 in the $(i,j)$-entry.  As $\dim W \ge 2$, there is $\ell \le w$, $\ell \ne j$ and $E_{\ell j} \in \sl(W) \subseteq L_0$.  So $[E_{\ell j}, E_{ij} - c E_{ji}] = -E_{i\ell}$ is in $L$, hence in $L_0$, yet $E_{i\ell} W \not \subseteq W$, a contradiction.  Thus $W = V$, i.e., $L_1$ acts irreducibly on $V$ and $L_0 = \sl_n$.

\subsubsection*{Case: $L_0 = 0$} Suppose $L_0 = 0$.  If $\car k \ne 2$, then $L + \tor$ cannot be irreducible (Remark \ref{maxrank}), so assume $\car k = 2$.  We prove \eqref{qr.symm}.  

Define $\hat{L}$ to be the subspace of $\gl(V)$ generated by $\tor$ and those $\g_{\pm \alpha}$ with nonzero intersection with $L$.  It is closed under the bracket.  Indeed, fixing nonzero elements $x_\alpha \in \g_\alpha$ for all $\alpha \in \Phi$, those $\g_{\pm \alpha}$ that meet $L$ are spanned by an element $x_\alpha + c_\alpha x_{-\alpha}$ for some $c_\alpha \in \kx$.  If $\g_{\pm \beta}$ also meets $L$, then
\[
[x_\alpha + c_\alpha x_{-\alpha}, x_\beta + c_\beta x_{-\beta}] \in \g_{\pm (\alpha + \beta)} + \g_{\pm (\alpha - \beta)}.
\]
As $L$ acts irreducibly on $V$, so does $\hat{L}$, and Lemma \ref{L0.subsystem} gives that $\hat{L} = \gl_n$ and in particular $\g_{\pm \alpha}$ meets $L$ for every root $\alpha$.

For each simple root $\alpha_i$, set $h_i \!: \Gm \to \GL_n$ to be a cocharacter such that $\alpha_j \circ h_i \!: \Gm \to \Gm$ is $t \mapsto 1$ if $i \ne j$ and $t \mapsto t^{r_i}$ for some $r_i \ne 0$ if $i = j$.  As 
\[
\Ad(h_i(t)) (x_{\alpha_i} + c_{\alpha_i} x_{-\alpha_i}) = t^{r_{i}} x_{\alpha_i} + \frac{c_{\alpha_i}}{t^{r_i}} x_{-\alpha_i},
\]
there is a $t_i \in \kx$ for each $i$ so that $\Ad(h_i(t_i)) (\g_{\pm \alpha_i} \cap L)$ is generated by $E_{i,i+1} + E_{i+1,i}$.  Conjugating $L$ by $\prod h_i(t_i)$ arranges this for all simple roots $\alpha_i$ at once, and it follows that the resulting conjugate of $L$ consists of symmetric matrices and intersects  $\g_{\pm \alpha}$ nontrivially for all $\alpha \in \Phi$, whence
$L$ contains the space of alternating matrices.
\end{proof}

\section{Type $A$ and   $\car k \ne 2$}

Recall that $\sl_n$ for $n \ge 2$ is either simple ($\car k$ does not divide $n$) or has a unique nontrivial ideal, the center (consisting of the scalar matrices, in case $\car k$ does divide $n$).  

The next two items have no restrictions
on the characteristic of $k$.  We do not need the first result in characteristic $2$. 
 
\begin{eg} \label{regular.A}
Suppose that  $x$ is regular nilpotent in $\sl_n$ for some $n \ge 2$; we claim that $e(x) = 2$, i.e., 2 generic $\SL_n(k)$-conjugates of $x$ generate $\sl_n$.  Up to conjugacy, $x$ has 
1's on the superdiagonal and 0's in all other entries. Choose a
conjugate $y$ of $x$ whose only nonzero entries are $x_2, \ldots, x_n$ on the subdiagonal.    Then  $w:=[x,y]$ is diagonal with entries $z_1, \ldots, z_n$
where $(z_1, \ldots, z_n) = (-x_2, x_2 - x_3, \ldots, x_{n-1} - x_n, x_n)$.   For a nonempty
open subvariety of $(x_2, \ldots, x_n)$
 the $z_i-z_j$ are distinct.   Thus, the algebra generated by $w$ and $x$ contains all
the positive simple root algebras and similarly the algebra generated by $w$ and $y$ contains all the negative
simple root algebras, whence $\langle x, y \rangle = \sl_n$.  Since the condition on generating $\sl_n$ is open (Lemma \ref{gen}\eqref{opengen.trivial}),
this implies that $2$ generic conjugates of $x$ generate $\sl_n$.  
\end{eg}

For $x \in \sl_n$, put $\alpha(x)$ for  the dimension of the largest eigenspace.

\begin{lem} \label{dc1}
For non-central $x \in \sl_n$ with $n \ge 2$, if $e > \frac{n-1}{n-\alpha(x)}$, then the subalgebra of $\sl_n$ generated by $e$ generic conjugates of $x$ fixes no $1$-dimensional subspace nor codimension-$1$ subspace of the natural module.
\end{lem}

The hypothesis that $x$ is non-central ensures that the denominator $n - \alpha(x)$ is not zero.

\begin{proof}
Suppose the subalgebra generated by $e$ generic conjugates of $x$ fixes a line.  Then (Lemma \ref{gen}\eqref{opengen.inv}), every subalgebra generated by $e$ conjugates fixes a line.  Putting $X := x^{\SL_n}$, there is a map $G \times (\times^e X) \to \times^e X$ via $(g, x_1, \ldots, x_e) \mapsto (\Ad(g)x_1, \ldots, \Ad(g)x_e)$, and by hypothesis $\times^e X$ belongs to the image of $G \times (\times^e (X \cap \p))$
where $\p$ is the stabilizer of the first basis vector in the natural module, the Lie algebra of a parabolic subgroup $P$ of $\sl_n$.  Thus
\[
e \cdot \dim X \le \dim \P^1 + e \cdot \dim(X \cap \p).
\]
and consequently
\begin{equation} \label{dc1.1}
e(\dim X - \dim (X \cap \p)) \le \dim (G/P) = n - 1.
\end{equation}
 
Now consider the variety $Y \subset X \times \P^{n-1}$ with $k$-points 
\[
Y(k) = \{ (y, \omega) \in X(k) \times \P(k^n) \mid y\omega = \omega \}.
\]
The projection of $Y$ on the first factor maps $Y$ onto $X$ with fibers of dimension $\alpha(x) - 1$.  The projection of $Y$ on the second factor maps $Y$ onto $\P^{n-1}$ with fibers of dimension $\dim(X \cap \p)$.  Consequently,
\[
\dim X + \alpha(x) - 1 = \dim Y = (n-1) + \dim(X \cap \p).
\]
Combining this with \eqref{dc1.1} gives $e \le \frac{n-1}{n-\alpha(x)}$.

Now suppose each subalgebra $\g$ generated by $e$ generic conjugates of $x$ fixes a codimension-1 subspace $V$ of the natural module.  Using the dot product we may identify the natural module $k^n$ with its contragradient $(k^n)^*$, and it follows that the subalgebra $\{ y^\T \mid y \in \g \}$ fixes the line in $(k^n)^*$ of elements vanishing on $V$.  Consequently $e \le \frac{n-1}{n-\alpha(x^\T)}$.  As $\alpha(x^\T) = \alpha(x)$, the claim is proved.
\end{proof}

\begin{prop} \label{A.gen}
Assume $\car k \ne 2$.  For each nonzero nilpotent $x \in \sl_n$ with $n \ge 3$, $e$ generic conjugates of $x$ generate $\sl_n$, where:
\begin{enumerate}
\item \label{A.gen.3} $e = 3$ if $x$ has Jordan canonical form with partition $(2, 2, \ldots, 2)$ or $(2, 2, \ldots, 2, 1)$.
\item \label{A.gen.2} $e = 2$ if $\alpha(x) \le \lceil n/2 \rceil$ but we are not in case \eqref{A.gen.3}.
\item \label{A.gen.o} $e = \lceil \frac{n}{n-\alpha(x)} \rceil$ if $\alpha(x) > \lceil n/2 \rceil$.
\end{enumerate}
\end{prop}

\begin{proof}
The conjugacy class of $x$ is determined by its Jordan form, which corresponds to a partition $(p_1, \ldots, p_\alpha)$ of $n$, i.e., a list of numbers $p_1 \ge p_2 \ge \cdots \ge p_\alpha > 0$ such that $p_1 + \cdots + p_\alpha = n$.  If $x$ has partition $(n)$, then $e(x) = 2$ by Example \ref{regular.A}.

If $x$ has partition $(2, 1, \ldots, 1)$, i.e., the Jordan form of $x$ has a unique nonzero entry, then $x$ generates a root subalgebra, and we may assume it corresponds to a simple root.  The other root subalgebras for simple roots and for the lowest root suffice to generate $\sl_n$, so in this case $n = \lceil n / (n-(n-1)) \rceil$ conjugates suffice to generate.

Thus we may assume that $n \ge 4$.

Suppose first that $x$ has partition $(2, 2, \ldots, 2)$ and view $x$ as the image of a regular nilpotent in $\sl_2$ under the diagonal embedding in $\sl_2^{\times n/2} \subset \sl_n$.  As in Example \ref{regular.A}, two $\SL_2^{\times n/2}$-conjugates suffice to generate $\sl_2^{\times n/2}$.  As the adjoint representation of $\sl_n$ restricts to a multiplicity-free representation of $\sl_2^{\times n/2}$, there are only a finite number of Lie algebras lying between $\sl_2^{\times n/2}$ and $\sl_n$.  Now $x^{\SL_n}$ generates $\sl_n$ as a Lie algebra, so it is not contained in any of these proper sublagebras and the irreducible variety $x^{\SL_n}$ is not contained in the union of the proper subalgebras.  This proves the claim that 3 conjugates suffice to generate $\sl_n$.

If $x$ has partition $(2, 2, \ldots, 2, 1)$, then we view it as the image of $x' \in \sl_{n-1}$ where $x'$ has partition $(2, 2, \ldots, 2)$,  for which three $\SL_{n-1}$-conjugates generate $\sl_{n-1}$.  That is, three generic $\SL_n$-conjugates of $x$ generate a subalgebra $\lsub$ that is quasi-regular (Example \ref{sln.qr}).  Moreover, as $n = 2\alpha-1$, $\lsub$ does not fix a 1-dimensional or codimension-1 subspace of the natural module (Lemma \ref{dc1}), and therefore $\lsub$ acts irreducibly and $\lsub$ is the whole algebra $\sl_n$ (Remark \ref{maxrank}).

Now suppose $\alpha(x) \le n/2$ and we are not in case \eqref{A.gen.3}.  Then $p_1 \ge 3$ and by passing to a nilpotent element in the closure of $x^{\SL_n}$ as in \S\ref{deforming}, we can reduce to the cases
\begin{enumerate}
\renewcommand{\theenumi}{\alph{enumi}}
\item \label{A.gen.even}   $n$ is even and $x$ has partition $(3, 2, \ldots, 2, 1)$; or
\item \label{A.gen.odd}  $n$ is odd and $x$ has partition $(3, 2, \ldots, 2)$.
\end{enumerate}

In case \eqref{A.gen.even}, we see by induction that
  we can generate $\sl_{n-1}$ with two $\SL_n$-conjugates and we argue as in the preceding case.
  
In case \eqref{A.gen.odd}, deform to $y \in \overline{x^{\SL_n}}$ with partition $(3, 2, \ldots, 2, 1, 1)$.  It is the image of $y' \in \sl_{n-1}$ with partition $(3, 2, \ldots, 2, 1)$.  By induction on $n$, two $\SL_{n-1}$-conjugates of $y'$ generate a copy of $\sl_{n-1}$.  Arguing as in the preceding cases concludes the proof of \eqref{A.gen.2}.

Finally, suppose $\alpha(x) > \lceil n/2 \rceil$, so in particular $p_\alpha = 1$.  Put $x' \in \sl_{n-1}$ for a nilpotent with partition $(p_1, \ldots, p_{\alpha-1})$.  By induction, we find that $\lceil n/(n-\alpha) \rceil$ $\SL_{n-1}$-conjugates suffice to generate a copy of $\sl_{n-1}$, and we complete the proof as before.
\end{proof}

\begin{cor} \label{A.odd}
Suppose $\car k \ne 2$.
For noncentral $x \in \gl_n$ with $n \ge 2$ such that $x^{[p]} \in \{ 0, x \}$, there exist $e > 0$ and elements $x_1, \ldots, x_e \in x^{\SL_n}$ that generate a subalgebra containing $\sl_n$ such that $e \cdot \dim x^{\SL_n} \le \frac94 n^2$.
\end{cor}

\begin{proof}
Suppose first that $x^{[p]} = 0$.  If $n = 2$, then $e(x) = 2$ by Example \ref{regular.A} and $\dim x^{\SL_n} = 2$, so assume that $n \ge 3$.  We consider the three cases in Proposition \ref{A.gen}.  In case \eqref{A.gen.3}, we have $\dim x^{\SL_n} \le n^2/2$ and $e(x) = 3$, so the claim is clear.  In case \eqref{A.gen.2}, $e = 2$ and $\dim x^{\SL_n} < n^2$.  In case \eqref{A.gen.o}, among those nilpotent $y$ with rank $n - \alpha(x)$, the one with the largest $\SL_n$-orbit has partition $(n-\alpha(x)+1, 1, \ldots, 1)$, whose orbit has dimension $n^2 - n - \alpha(x)^2 + \alpha(x)$.  Consequently, 
\[
e(x) \cdot \dim x^{\SL_n} < (n+\alpha(x) - 1)(2n - \alpha(x)).
\]
This is a quadratic polynomial in $\alpha(x)$ opening downwards with maximum at $(n+1)/2$.  As $\alpha \ge \lceil n/2 \rceil + 1 \ge n/2 + 1$, the right side is no larger than $\frac94 n^2 - 3n/2$ verifying the claim for $x$ nilpotent.

For $x \in \sl_n$ noncentral toral, let $y$ be the nilpotent element provided by Lemma \ref{deform}.  Then $\dim x^G = \dim y^G$ and the same number of conjugates suffice to generate a subalgebra containing $\sl_n$, as in Lemma \ref{gen}\eqref{opengen.1} with $M = \sl_n$ and $N = \z(\sl_n)$.
\end{proof}

\section{Type $A$ and $\car k = 2$}

\begin{prop}   \label{A.2.gen}
Suppose $\car k = 2$ and let $x \in \sl_n$ with $n \ge 2$ be a nilpotent element of square $0$ and rank $r$.   
  Then $\sl_n$ can be generated by $e:= \max \{ 3, \lceil{n/r}\rceil  \}$
  conjugates of $x$.
  \end{prop}
  
  \begin{proof}    Note the result is clear if $x$ is a root element by taking root elements in each of the simple
  positive root subalgebras and in the root subalgebra corresponding to the negative of the highest root.
  This gives the result for $n=2,3$ and shows that for $n=4$, it suffices to consider $r=2$.   Choose  two
  conjugates of $x$ and $y$ generating $\sl_2 \times \sl_2$.   It is straightforward to see for a generic conjugate
  $z$ of $x$, the elements $x$, $y$ and $z$ generate $\sl_4$. So assume $n > 4$.   
  
  If $n$ is odd, it follows by induction on $n$ that $e$ conjugates of $x$ can generate
  an $\sl_{n-1}$.   On the other hand, the condition on the rank implies by Lemma \ref{dc1} that $e$ generic conjugates of $x$ do not
  fix a $1$-space or a hyperplane.   Thus,  generically $e$ conjugates of $x$ generate a subalgebra that acts irreducibly (as in the proof of Lemma \ref{dc1}) and is quasi-regular by Example \ref{sln.qr}.
  Also, we see that generically the dimension of the Lie algebra generated by
  $e$ conjugates has dimension at least $(e-1)^2 - 1$.  Since $n > 4$, this is larger than the dimension of the space
  of symmetric matrices, whence by Theorem \ref{qr.thm},  we see that $e$ generic conjugates generate $\sl_n$.
  
  Now assume that $n$ is even.   By passing to closures we may assume that $r < n/2$ (since $n > 4$,  $e = 3$
  for both elements of rank $n/2$ and rank $n/2-1$).   Now argue just as for the case that $n$ is odd.
  \end{proof}
  
\begin{rmk*} 
The result also holds for idempotents (i.e., toral elements) of rank $e \le n/2$ by a closure argument.
\end{rmk*}
  
  \begin{cor} \label{A.2}  Suppose $\car k = 2$.
  For noncentral $x \in \gl_{n}$ with $n \ge 2$ such that $x^{[2]} \in \{ 0, x \}$, there exist $e > 0$ and elements $x_1, \ldots, x_e \in x^{\SL_n}$ that generate a subalgebra containing $\sl_n$ such that $e \cdot \dim x^{\SL_n} \le 2n^2 - 2$.
  \end{cor}
  
  \begin{proof}   Let $x \in \sl_n \setminus \z(\sl_n)$ satisfy $x^{[2]} = 0$ and put $r$ for the rank of $x$.  Then $\dim x^{\SL_n} =  n^2 - (r^2 + (n-r)^2)=  2r(n-r)$.   If $3$ conjugates of $x$ generate $\sl_n$, then $3 \cdot \dim x^{\SL_n} = 6r(n-r)$.  This has a maximum at $r = n/2$, where it is $\frac32 n^2 \le 2n^2 - 2$.  Otherwise $(n+r)/r$ conjugates suffice to generate, and we have $e \dim x^{\SL_n} \le 2(n^2-r^2) \le 2n^2 -2$.

  Now suppose that $x \in \gl_n$ is noncentral toral.  Take $y \in \overline{x^{\Gm \GL_n}}$ such that $y^{[2]} = 0$ as in Lemma \ref{deform}, so $\dim y^{\SL_n} = \dim x^{\SL_n}$.  Applying Lemma \ref{gen}\eqref{opengen.1} with $M = \sl_n$ and $N = \z(\sl_n)$ gives that $e \cdot \dim x^{\SL_n} \le 2n^2 - 2$ also in this case.
   \end{proof}

\section{Type $C$ and $\car k \ne 2$}

\begin{prop} \label{C.gen}
Assume $\car k \ne 2$.  For every nonzero nilpotent $x \in \sp_{2n}$ for $n \ge 1$ of rank $r$, $e$ generic conjugates of $x$ generate $\sp_{2n}$, where:
\begin{enumerate}
\item \label{C.gen.3} $e = 3$ if $x$ has Jordan canonical form with partition $(2, 2, \ldots, 2)$.
\item \label{C.gen.2} $e = 2$ if $r \ge n$ but we are not in case \eqref{C.gen.3}.
\item \label{C.gen.o} $e = 2 \lceil n/r \rceil$ if $r < n$.
\end{enumerate}
\end{prop}

\begin{proof}
The conjugacy class of $x$ is determined by its Jordan form, which corresponds to a partition $(p_1, \ldots, p_\alpha)$ of $2n$ with $p_1 \ge p_2 \ge \cdots \ge p_\alpha$ such that odd numbers appear with even multiplicity.  Note that $\sp_2 = \sl_2$, so the $n = 1$ case holds by Example \ref{regular.A}.

By specialization (replacing $x$ with an element of $\overline{x^{\Sp_{2n}}}$ as in \S\ref{deforming}), we may replace in the partition of $x$
\begin{equation} \label{C.gen.eq2}
(2s+2,1,1) \leadsto (s+1, s+1, 2) \  \text{or} \ (2s+1, 2s+1, 1, 1) \leadsto (2s, 2s, 2, 2) \text{\ for $s \ge 2$}
\end{equation}
without changing the rank $r$ of $x$ nor whether the partition is $(2, \ldots, 2)$.  In this way, we may assume that $p_\alpha \ge 2$ or $p_1 \le 4$.

\subsubsection*{Case \eqref{C.gen.3}} Suppose that $x$ has partition $(2, 2, \ldots, 2)$.  Two conjugates of $x$ suffice to generate a copy of $\sl_2^{\times n} \subset \sp_{2n}$, and this contains a regular semisimple element of $\sp_{2n}$.  Furthermore, the natural representation of $\sp_{2n}$ is multiplicity-free for $\sl_2^{\times n}$, so one further conjugate suffices to produce a subalgebra that is irreducible on the natural module.  Appealing to Remark \ref{maxrank}, the claim follows in this case.

\subsubsection*{Case $\sp_4$} For the case $n = 2$, it remains to consider $x$ with partition $(4)$, i.e., a regular nilpotent.  A pair of generic conjugates generates an irreducible subalgebra. By passing to $(2,2)$, we see it also generically contains an element as in \eqref{sr}, whence the result. 

\subsubsection*{Case $\sp_6$}
Suppose $x \in \sp_6$; it suffices to assume that $x$ has rank at least $3$
and $p_1 \ge 3$.  We want to show that two conjugates of $x$ can generate.
By passing to closures, it suffices to assume that $x$ is nilpotent with partition 
$(4,1,1)$.   As in \eqref{C.gen.eq2}, the closure of the class of $x$ contains the class
corresponding to the partition $(2,2,2)$.   Since two conjugates of the latter 
can generate an $\sl_2 \times \sl_2 \times \sl_2$, we see via Lemma \ref{gen}\eqref{opengen.2} that generically two
conjugates of $x$ generate a Lie algebra containing a strongly regular semisimple
element and so a quasi-regular algebra.

By the $\sp_4$ case, we see that generically the largest composition factor of
the algebra generated by two conjugates of $x$ is at least $4$-dimensional and, by the paragraph
above, the smallest is at least $2$-dimensional.   Thus, for  generic $y \in x^G$, the subalgebra $\langle x, y \rangle$  generated by $x$ and $y$
is either irreducible or the module is a direct sum of nondegenerate spaces of dimension
$4$ and $2$.  However, this would imply that $x$ and $y$ would be trivial on the two
dimensional space, a contradiction.   Thus, a generic pair of conjugates of $x$ and $y$
generates an irreducible quasi-regular subalgebra.   Since we are in characteristic
different from $2$, this implies that generically $\langle x, y \rangle = \sp_6$
as required.

\subsubsection*{Case $2n \ge 8$ and $x$ has partition $(3, 3, 2, \ldots, 2, 1, 1)$}  Suppose now that $x$ has partition $(3, 3, 2, \ldots, 2, 1, 1)$ so $r = n$.  
We know that two conjugates of $x$ do not generically fix a $1$-sapce.
 By induction, two conjugates can generate $\sp_{2n-2} \times \sp_2$ and in particular generically the algebra contains a strongly regular element as defined in Example \ref{sr.eg}.
If $n$ is even, we can also generate $\sl_n$ and so the smallest composition factor has dimension $n$, whence the algebra is generically
irreducible and the result follows.  In either case ($n$ even or odd), we see that by induction
one can generate the subalgebra of the stabilizer of a singular $1$-space containing
the Levi subalgebra as well as the central root subalgebra.   This shows that  there
exist a pair of 
a conjugates of $x$  that do not preserve a $2$-space.    Thus, generically
the smallest composition for the algebra generated by two conjugates of $x$
is at least four dimensional.  Since generically there is a composition factor of 
dimension at least $2n-2$, this implies that generically two conugates of 
$x$ generate a quasi-regular irreducible subalgebra and so the full algebra.

\subsubsection*{Case $r \ge n$} We now consider the case where $r \ge n$ (and $2n \ge 8$).

If $p_\alpha \ge 2$, then, as $\alpha = 2n - r \le n$ and we are not in case \eqref{C.gen.3}, we may replace $2s \leadsto (s,s)$ for $s \ge 3$, $(s,s) \leadsto (s-1, s-1, 1, 1)$ for $s \ge 4$, or $(4, 2) \leadsto (3, 3)$ as long as we retain the property that $\rank x \ge n$.  In this way, we may assume that $p_\alpha \le 1$ or $p_1 \le 3$.

So suppose $p_\alpha = 1$, in which case we may assume that $p_1 \le 4$.  We may replace $(4, 4, 1, 1) \leadsto (3, 3, 2, 2)$, $(4, 2) \leadsto (3, 3)$, or $(4,3,3,1,1) \leadsto (3,3,2,2,2)$ without changing the rank of $x$.  Repeating these reductions and those in the previous paragraph, we are reduced to considering partitions $(4, 1, \ldots, 1)$ of rank 3 (excluded because $r \ge n \ge 4$) or $p_1 = 3$.

If there are at least four 3's, we substitute $(3^4, 1^2) \leadsto (3^2, 2^3, 1^2)$ if $p_\alpha = 1$ or $(3^4) \leadsto (3^2, 2^3)$ if $p_\alpha > 1$.  Thus we may assume that $x$ has partition $(3^2, 2^{r-4}, 1^t)$.  As $2r \ge 2n = 2r - 2 + t$, we find that $x$ has partition $(3^2, 2^{r-4}, 1^2)$ with $r = n$ (in which case the proposition has already been proved) or partition $(3^2, 2^{r-4}$) with $r = n + 1$, which specializes to the previous case.

\subsubsection*{Case $r < n$}  Now suppose that $x$ has rank $r < n$, so in particular $p_\alpha = 1$ and we may assume that $p_1 \le 4$.  Specializing as in \eqref{C.gen.eq2} also with $s = 1$, we may assume that 
$x$ has partition $(2^r, 1^{2n-2r})$.
If $r = 1$, then $2n$ conjugates suffice to generate $\sp_{2n}$ by, for example, \cite{CSUW}.  So assume $r \ge 2$.

Clearly, $n/r \le n/2 < n-r$, so there are at least $2v+2$ 1-by-1 Jordan blocks in $x$ for $e := 2\lceil n/r \rceil = 2v + 2$.  We then subdivide $x$ into two blocks on the diagonal, with partitions $(2, 1^{2v})$ and $(2^{r-1}, 1^{2n-2r-2v})$.  By the $r= 1$ case, $e$ generic conjugates of the first generate an $\sp_e$ subalgebra and by induction $\max \{ 3, 2\lceil (n - v-  1)/ (r-1) \rceil \}$ conjugates of the second generate an $\sp_{2n-e}$ subalgebra.  As $2n \le re$, we have $(n-v-1)/(r-1) \le n/r$, and the max in the preceding sentence is at most $e$.  Note that $\sp_e \times \sp_{2n-e}$ contains a regular semisimple element of $\sp_{2n}$ and the natural module has composition factors of size $e$, $2n-e$.

Alternatively, we may subdivide $x$ into blocks with partitions $(2^r, 1^{2n-2r-2})$ and $(1^2)$.  By induction, $e$ generic conjugates of this element give an $\sp_{2n-2}$ subalgebra, with composition factors of size 1, 1, $2n-2$.  As this list does not meet the list of composition factors from the previous paragraph, the generic subalgebra generated by $e$ conjugates acts irreducibly on the natural module, and we are done via an application of Remark \ref{maxrank}.
\end{proof}

\begin{cor} \label{C.odd}
Assume $\car k \ne 2$.  For nonzero nilpotent or noncentral semisimple $x \in \sp_{2n}$ with $n \ge 1$, there exist $e > 0$ and elements $x_1 \ldots, x_e \in x^{\Sp_{2n}}$ that generate $\sp_{2n}$ such that $e \cdot \dim x^{\Sp_{2n}} \le 6n^2$. 
\end{cor}

\begin{proof}
Note that we are done if 3 conjugates of $x$ suffice to generate $\sp_{2n}$, as $\dim x^G \le 2n^2$.  Moreover, the case $n = 1$ holds by Corollary \ref{A.odd}.

Recall that $\alpha(x)$ is the dimension of the largest eigenspace of $x$ (and so for $x$ nilpotent, the rank of $x$
is $2n - \alpha(x)$).

\subsubsection*{Nilpotent case} Suppose that $x$ is nonzero nilpotent and put $e(x)$ for the minimal number of 
conjugates of $x$ needed to generate $\sp_{2n}$.  
We may assume that $e(x) > 3$ and so $r < n$.   In particular, $\alpha(x) > n$. 

We have $e(x) \le 2\lceil \frac{n}{2n-\alpha(x)}\rceil$ by Proposition \ref{C.gen}. To bound $\dim x^G$, we replace $x$ with $y$ such that $\alpha(y) = \alpha(x)$ and $y$ specializes to $x$, i.e., $x$ belongs to the closure of $y^{\Gm G}$.  Then $\sp_{2n}$ is also generated by $2\lceil \frac{n}{2n-\alpha(x)} \rceil$ conjugates of $y$ and $\dim x^G \le \dim y^G$.  The element $x$ is given by a partition $(p_1, \ldots, p_\alpha)$ as in the proof of Proposition \ref{C.gen}.  

We claim that $y$ can be taken to have partition $(2s, 2, 1^{\alpha(x) - 2})$ or $(2s, 1^{\alpha(x) -  1})$.  Indeed, let $I := \{ i \mid \text{$i > 1$ and $p_i > 2$} \}$.  Then the element $y$ with partition $(p'_1, p'_2, \ldots, p'_\alpha)$ where
\[
p'_i = \begin{cases}
2 & \text{if $i \in I$} \\
p_i & \text{if $i > 1$ and $i \not\in I$} \\
p_1 + \sum_{i \in I} (p_i - 2) & \text{if $i = 1$}
\end{cases}
\]
specializes to $x$, compare \cite[3.10]{Hesselink:sing} or \cite[6.2.5]{CMcG}.  Replacing $x$ with $y$ we find an element with partition $(2s, 2^r, 1^{\alpha(x) - r - 1})$ for some $s \ge 1$ and some $r$.  If $r \ge 2$ and $s > 1$, then we may replace $x$ with an element with partition $(2s+2, 2^{r-2}, 1^{\alpha(x)-r+1})$ and repeating this procedure gives the claim.  

The formula for $\dim C_{\Sp_{2n}(k)}(y)$ in \cite[p.~39]{LiebeckSeitz} gives that it is at least $n + (\alpha(x)^2 - 1)/2$.  Applying $\lceil n/(2n-\alpha(x)) \rceil < (3n - \alpha(x))/(2n-\alpha(x))$, we find that $e(x) \cdot \dim x^G < 6n^2 + \alpha(x) (n - \alpha(x)) + 1/(2n-\alpha(x))$.  As $n - \alpha(x)$ is negative, we have verified the required inequality for $x$ nilpotent.

\subsubsection*{Semisimple case}
We may assume $x$ is diagonal.  Put $\alpha_0$ for the number of nonzero entries in $x$; we will construct a nilpotent $y$ in the closure of $x^{\Gm \Sp_{2n}}$.  Recall that the diagonal of $x$ consists of pairs $(t, -t)$ with $t \in k$.

Suppose first that $\alpha_0 \ge n$.  We pick $y$ to be block diagonal as follows.  For a 4-by-4 block with entries $(0,0,t,-t)$ for some $t \in \kx$, we make a 4-by-4 block in $y$ in the same location, where the 2-by-2 block in the upper right corner is generic for $\sp_4$.  As $\alpha_0 \ge n$, by permuting the entries in $x$ we may assume that all pairs $(0,0)$ on the diagonal of $x$ are immediately followed by a $(t, -t)$ with $t \ne 0$.  Thus, it remains to specify the diagonal blocks in $y$ at the locations corresponding to the remaining 2-by-2 blocks $(t, -t)$ for $t \ne 0$ in $x$, for which we take $y$ to have a 1 in upper right corner.  We have constructed a nilpotent $y$ with $\rank y \ge n$, so $e(x) \le e(y) \le 3$ by Prop.~\ref{C.gen}, and $e(x) \cdot \dim x^{\Sp_{2n}} \le 6n^2$.

Now suppose $\alpha_0 < n$.  Let $x_0$ be a $2\alpha_0$-by-$2\alpha_0$ submatrix consisting of all the nonzero diagonal entries in $x$ together with $\alpha_0$ zero entries.  Take $y_0$ to be the nilpotent element constructed from $x_0$ as in the preceding paragraph, and extend it by zeros to obtain a nilpotent $y$ with $\alpha(y) = 2n-\alpha_0 > n$.  Then $y$ is in the closure of $x^{\Gm \Sp_{2n}}$ and $e(y) \le 2 \lceil n/\alpha_0 \rceil < 2(n+\alpha_0)/\alpha_0$.  On the other hand, the centralizer of $x$ has dimension at least $\dim \Sp_{2n-\alpha_0} + \alpha_0/2 = 2n^2 - 2n\alpha_0 + \alpha_0^2/2 + n$.  Thus $\dim x^{\Sp_{2n}} \le 2n\alpha_0 - \alpha_0^2/2$.  In summary,
$e(x) \cdot \dim x^{\Sp_{2n}} < (n+\alpha_0)(4n - \alpha_0) = 4n^2 + 3\alpha_0 n - \alpha_0^2$.  As a function of $\alpha_0$, it is a parabola opening down with max at $\alpha_0 = 1.5n$, so its maximum for $\alpha_0 < n$ is where $\alpha_0 = n-1$, i.e., the max is at most $6n^2 - n -1$.
\end{proof}

\section{Types $B$ and $D$ with $\car k \ne 2$}

\begin{prop} \label{BD.gen}
Assume $\car k \ne 2$.  For every nonzero nilpotent $x \in \so_{n}$ for $n \ge 5$, $\max \{ 4, \lceil \frac{n}{n-\alpha(x)} \rceil \}$ conjugates of $x$ generate $\so_{n}$.
\end{prop}

\begin{proof}
The $O_n$-conjugacy class of $x$ is determined by its Jordan form, which is given by a partition $(p_1, \ldots, p_\alpha)$ of $n$ where even values occur with even multiplicity.  We go by induction on $n$.  As $\so_5 \cong \sp_4$, the $n = 5$ case is covered by Proposition \ref{C.gen}, which gives 4 as the largest number of conjugates needed to generate.  For $n = 6$, $\so_6 \cong \sl_4$, and this case is handled by Proposition \ref{A.gen}.  So assume $n \ge 7$.

Suppose first that the number $\delta$ of 1's in the partition for $x$ is at most 1.  Then we can find an element $y$ in the closure of $x^{\SO_n}$ with partition 
\begin{enumerate}
	\renewcommand{\theenumi}{\roman{enumi}}
	\item \label{BD.gen.0} $(2^{n/2})$ if $n \equiv 0 \bmod 4$;
	\item \label{BD.gen.1} $(2^{(n-1)/2}, 1)$ if $n \equiv 1 \bmod 4$; 
	\item \label{BD.gen.2} $(3^2, 2^{(n-6)/2})$ if $n \equiv 2 \bmod 4$; or
	\item \label{BD.gen.3} $(3, 2^{(n-3)/2})$ if $n \equiv 3 \bmod 4$.
\end{enumerate}
To see this, we specialize $(2s, 2s) \leadsto (s^4)$ for $s \ge 2$; $s \leadsto (s-4,2,2)$ for odd $s \ge 7$; or $(s,1) \leadsto ((s+1)/2, (s+1)/2))$ for odd $s \ge 3$ and $\delta = 1$.  Together with trivial reductions such as $(5^2) \leadsto (3^2, 2^2)$ brings us to a partition of the form $(3^b, 2^c, 1^\delta)$ for some $b \le 3$ and some $c$ from which the claim quickly follows.
For such a $y$, 2 conjugates suffice to generate a copy of $\sl_2^{\times n/2}$, $\sl_2^{\times (n-1)/2}$, $\so_3 \times \so_3 \times \sl_2^{\times (n-6)/2}$, or $\so_3 \times \sl_2^{(n-3)/2}$ respectively.  As in the proof of Proposition \ref{C.gen}, it follows that 3 conjugates are enough to generate $\so_n$.

Now suppose there are more 1's in the partition for $x$.  We specialize using 
\[
(2s+1, 1) \leadsto (s+1, s+1) \text{\ for $s \ge 1$} \quad \text{and} \quad (s,s,1,1) \leadsto (s-1,s-1,2,2) \text{\ for $s \ge 4$.}
\]
If, after a step in this specialization process, we find that only 0 or 1 1-by-1 blocks remain, we are done by the preceding paragraph.  Therefore, we may assume that $x$ has partition $(2^{2t}, 1^u)$ for $u \ge 2$.

Write out $t = 2t_0 + \delta$ for $\delta = 0$ or 1, and set $v = 2t_0 \lceil \frac{u}{2t} \rceil$.  We can view $x$ as block diagonal where the first block has partition $(2^{2t_0}, 1^v)$ and the second has partition $(2^{2t_0 + 2\delta}, 1^{u-v})$.  For the first block, 
\[
e := 2+\left\lceil \frac{v}{2t_0} \right\rceil = 2+ \left\lceil \frac{u}{2t} \right\rceil
\]
conjugates suffice to generate an $\so_{2t_0 e}$ subalgebra by induction on $n$.  For the second block, we note that
\[
\frac{u-v}{2t_0 + 2\delta} \le \frac{u}{2t},
\]
so, by induction, $e$ conjugates suffice to generate an $\so_{n-2t_0e}$ subalgebra.  Because $\so_{2t_0e} \times \so_{n-2t_0e}$ contains a regular semisimple element and the natural module has composition factors of size $2t_0e$ and $n-2t_0e$, we conclude as in the proof of Proposition \ref{C.gen} that $e$ conjugates of $x$ suffice to generate $\so_n$.
\end{proof}

\begin{cor} \label{BD.odd}
Assume $p := \car k \ne 2$.  For noncentral $x \in \so_n$ with $n \ge 5$ such that $x^{[p]} \in \{ 0, x \}$, there exist $e > 0$ and elements $x_1, \ldots, x_e \in x^{\SO_n}$ that generate $\so_n$  such that $e \cdot \dim x^{\SO_n} \le 2(n-1)^2$.
\end{cor}

\begin{proof}
As $\car k \ne 2$, we identify $\spin_n$ with $\so_n$ via the differential of the covering map $\Spin_n \to \SO_n$.  We argue as in the proof of Corollary \ref{C.odd}, replacing $\sp_{2n}$ with $\so_n$ and references to Proposition \ref{C.gen} with references to \ref{BD.gen}.  We may assume that $e(x) > 4$, for otherwise $e(x) \cdot x^{\SO_n} \le 4 \cdot \left( \binom{n}{2} - \lfloor n/2 \rfloor \right) \le 2(n-1)^2$.

\subsubsection*{Nilpotent case} Suppose that $x$ is nonzero nilpotent.  We have $e(x) \le \lceil n/(n-\alpha) \rceil$ and in particular we may assume that $\alpha > \frac23 n$.  Recall that the $O_n$-orbit of $x$ is determined by a partition $(p_1, \ldots, p_\alpha)$ of $n$, where even numbers appear with even multiplicity.  As in the proof of Corollary \ref{C.odd}, we may replace $x$ with $y$ with partition $(p_1 + \sum_{i=2}^\alpha (p_i-1), 1^{\alpha-1})$.  This element has $\alpha(y) = \alpha(x)$ and orbit of size $\binom{n}{2} - \binom{\alpha}{2}$.  As $e(x) < (2n-\alpha)/(n-\alpha)$, it follows that $e(x) \cdot x^{\SO_n} < \frac12 (2n-\alpha)(n+\alpha-1)$.  The upper bound is maximized for $\frac23 n \le \alpha \le n$ at the lower bound, where it is 
$\frac29 n(5n-3) < 2(n-1)^2$.

\subsubsection*{Semisimple case} Suppose that $x$ is noncentral diagonal in $\so_n$.  

Suppose first that $n$ is even.  If $\alpha_0 \ge n/2$, then pick $y$ as  in Corollary \ref{C.odd}, so $\alpha(y) = n/2$, $e(y) \le 4$, and we are done.  If $\alpha_0 < n/2$, we perform the same construction as in the last paragraph of the proof of \ref{C.odd} to obtain $y$ with $\alpha(y) = n - \alpha_0$, so $e(y) \le \max \{ 4, \lceil n/\alpha_0 \rceil \}$; suppose $\lceil n / \alpha_0 \rceil > 4$, i.e., $n/\alpha_0 > 4$, i.e., $\alpha_0 < n/4$.  The orbit of $x$ has dimension at least $\dim \SO_n - \dim \SO_{n-\alpha_0} - \alpha_0 / 2$, whence $e(x) \cdot \dim x^{\SO_n} < (n+\alpha_0) (n-\alpha_0/2 - 1)$, where the right side is maximized at $\alpha_0 = n/4$ and again we verify  that the upper bound is at most $2(n-1)^2$.

When $n$ is odd, we view $x$ as lying in the image of $\so_{n-1} \hookrightarrow \so_n$ and take $y$ in this same image as constructed by the method in the previous paragraph.  Computations identical to the ones just performed again verify $e(x) \cdot \dim x^{\SO_n} < 2(n-1)^2$.
\end{proof}

\section{Type $D$ with $\car k = 2$}

 
 \subsection*{Concrete descriptions} 
 For sake of precision, we first give concrete descriptions of the groups and Lie algebras associated with a nondegenerate quadratic from $q$ on a vector space $V$ of even dimension $2n$ over a field $k$ (of any characteristic).  The orthogonal group $\OO(q)$ is the sub-group-scheme of $\GL(V)$ consisting of elements that preserve $q$, i.e., such that $q(gv) = q(v)$ for all $v \in V\otimes R$ for every commutative $k$-algebra $R$; the special orthogonal group $\SO(q)$ is the kernel of the Dickson invariant $\OO(q) \to \Z/2$; and the groups of  similarities $\GO(q)$ and proper similarities $\SGO(q)$ are the sub-group-schemes of $\GL(V)$ generated by the scalar transformations and $\OO(q)$ or $\SO(q)$ respectively, see for example \cite[\S12 and p.~348]{KMRT} or  \cite[Ch.~IV]{Knus:qf}.  For $n \ge 3$, the group $\SO(q)$ is semisimple of type $D_n$, but neither simply connected nor adjoint.
 
 The statement that $q$ is nodegenerate means that the bilinear form $b$ on $V$ defined by $b(v,v') := q(v+v') - q(v) - q(v')$ is nondegenerate.  Viewing the Lie algebra of a group $G$ over $k$ as the kernel of the homomorphism $G(k[\eps]) \to G(k)$ induced by the map $\eps \mapsto 0$ from the dual numbers $k[\eps]$  to $k$, one finds that $\oo(q)$ is the set of $x \in \gl(V)$ such that $b(xv,v) = 0$ for all $v \in V$.  Since $\OO(q)/\SO(q) \cong \Z/2$, $\so(q) = \oo(q)$.  As $b$ is nondegenerate, the equation $b(Tv,v') = b(v, \sigma(T)v')$ defines an involution $\sigma$ on $\End(V)$.  The set of alternating elements $\{ T - \sigma(T) \mid T \in \End(V) \}$ is contained in $\so(q)$ and also has dimension $2n^2 - n$ \cite[2.6]{KMRT}, therefore the two subspaces are the same.   The Lie algebra $\go(q)$ of $\GO(q)$ and $\SGO(q)$ is the set of elements $x \in \gl(V)$ such that there exists a $\mu_x \in k$ so that $b(xv,v) = \mu_x q(v)$ for all $v$.   It has dimension one larger than $\so(q)$. 
 
 We assume for the remainder of the section that $\car k = 2$. 
 
\begin{eg} \label{D.2n.eg}
When $V = k^{2n}$ and $q$ is defined by $q(v) = \sum_{i=1}^n v_i v_{i+n}$, we write $\so_{2n}$ instead of $\so(q)$, etc.  The linear transformation $x$ obtained by projecting on the first $n$ coordinates satisfies $b(xv,v) = q(v)$ for all $v \in V$, so it and $\so_{2n}$ span $\go_{2n}$.

Suppose $x \in \go_{2n}$ is a projection, i.e., $x^2 = x$, so $x$ gives a decomposition $k^{2n} = \ker x \oplus \im x$ as vector spaces.  If $x$ belongs to $\so_{2n}$, then this is an orthogonal decomposition and $b$ is nondegenerate on $\ker x$ and $\im x$. Up to conjugacy, $x$ stabilizes the subspaces spanned by vectors with nonzero entries only in the first $n$ coordinates or the last $n$ coordinates, which exhibits $x$ as the image of some toral $\hat{x}$ under an inclusion $\gl_n \hookrightarrow \so_{2n}$ such that $2 \rank \hat{x} = \rank x$.  Suppose $x \not\in \z(\so_{2n})$, so $\hat{x} \not\in \z(\gl_n)$.  Let $\hat{y} \in \gl_n$ denote the nilpotent obtained for $\hat{x}$ as in Lemma \ref{deform}, and put $y \in \so_{2n}$ for its image.  Then $y$ is in the closure of $x^{\Gm \SO_{2n}}$ and $\rank y \le \rank x$ with equality if $\rank x \le n$.

If $x \in \go_{2n} \setminus \so_{2n}$ has $x^2 = x$, then $\im x$ and $\ker x$ are maximal totally isotropic subspaces.  To see this, note that if $q(v) \ne 0$, then $b(xv,v) = \mu_x q(v) \ne 0$, which is impossible if $xv \in \{ 0, v\}$.
\end{eg}

We consider how many conjugates of an $x \in \so_{2n}$ with $x^{[2]} \in \{ 0, x \}$ suffice to generate a subalgebra of $\so_{2n}$ containing the derived subalgebra $[\so_{2n}, \so_{2n}]$.  We apply Lemma \ref{gen}\eqref{opengen.1} with $G = \GO_{2n}$, $M = [\so_{2n}, \so_{2n}]$, and $N = \z(M)$, so $\dim N = 0$ or 1, $\dim M/N \ge 2n^2 - n - 2$ and $\dim \g/M = 2$.

\begin{eg} \label{D.2.comp}
One can verify by computing with an example that for $x \in \so_{2n}$ with $x^{[2]} = 0$, $e$ conjugates suffice to generate $[\so_{2n}, \so_{2n}]$ in the cases (a) $x$ is a root element and $n = e = 4$ or 5 or (b) $n = 7$ or 8, $e = 4$, and $x$ has rank 4.  (In the last case, note that $x$ can be taken to have Jordan form with partition $(2^4, 1^{2n-4})$.)  Magma code is provided with the arxiv version of this paper.
\end{eg}

\begin{lem}   \label{D.2.rt}
Let $\g = \so_{2n}$ with $n \ge 4$.  If $x$ is a root element,
 and  $m \ge 4$, then $m$ generic conjugates of $x$ generate the derived subalgebra
of a natural $\so_{2m}$.
\end{lem}

\begin{proof}  The case $n = 4$ is from Example \ref{D.2.comp}.

Now assume that $n > 4$ and $4 \le m < n$.     By induction on $n$, we know
that $m$ conjugates can generate the derived subalgebra of a copy of $\so_{2m}$.  Clearly any $m$ conjugates
have a fixed space of dimension at least $2n-2m$ and generically this space
will be nondegenerate, whence this $\so_{2m}$ is naturally embedded in $\so_{2n}$.

Now assume that $m=n$; by Example \ref{D.2.comp} we may assume that $n \ge 6$.  So now take  $n-2$ generic  root elements,  $x_1, \ldots, x_{n-2}$; 
they generate the derived subalgebra of a natural $\so_{2n-4}$ by induction.
Let us take a basis of $k^{2n}$ as in Example \ref{D.2n.eg}.  We identify our $\so_{2n-4}$ as the one acting trivially on the subspace spanned by $v_1, v_{n+1}, v_2, v_{n+2}$.

Then  consider two copies of the derived subalgebra of $\so_{2n-2}$  acting on the spaces  spanned by $v_i$ and $v_{n+i}$ for $1 \le i < n$ and for $1 < i \le n$.   These both contain our $\so_{2n-4}$ and by induction we can
choose $x$, $y$ respectively so that $x, x_1, \ldots, x_{n-2}$ generate the first copy of the derived subalgebra of $\so_{2n-2}$
and $x_1, \ldots, x_{n-2}, y$ generate the second copy.  These two copies generate the derived subalgebra of $\so_{2n}$, as can be seen by considering
the root elements in each one. 
\end{proof}

 \begin{prop}  \label{D.2.gen}
  Let $G = \SO_{2n}$ with $n \ge  3$ over an algebraically closed field $k$
 of characteristic $2$.   For noncentral $x \in \g$ such that $x^{[2]} \in \{ 0, x \}$, $\max\{4, \lceil n/r \rceil \}$ conjugates
of $x$ generate a Lie subalgebra containing
 $[\g, \g]$ where $2r$ is the codimension of the largest eigenspace of $x$.
 \end{prop}
 
 For $x$ as in the proposition: if (1) $x^{[2]} = x$ and $\rank x \le n$ or (2) $x$ is nilpotent, then $2r$ is the rank of $x$.  
 
 \begin{proof} Suppose $x^{[2]}=0$ and $n \ge 4$.
   The closure of $x^G$ contains
 a nilpotent element $y$ of the same rank with $y$ contained in
 a Levi subalgebra $\gl_n$ (by \cite[Table 4.1]{LiebeckSeitz}, this reduces
 to the case of $\so_4$ where the result is clear).   Thus 
 we may assume that $x$ is nilpotent
and is contained in  $\sl_n$.  The case where $x$ is a root element was considered in Lemma \ref{D.2.rt} (with $m = n$), so we may assume $r \ge 2$.

If $n/r \le 3$,  then by the result for $\sl_n$ (Prop.~\ref{A.2.gen}),  we can generate
an $\sl_n$ with $3$ conjugates.  Since $\g/\sl_n$ is multiplicity free as an $\sl_n$-module, this implies the result.

Suppose that $n \le 8$.  The result follows by the previous paragraph unless $r=2$ and $n=7$ or $8$.   These cases were settled in Example \ref{D.2.comp}.

Now suppose that $n \ge 9$ and put $e$ for the maximum appearing in the statement.   By the result for $\sl_n$, $e$ conjugates can
generate an $\sl_n$ and something containing the derived subalgebra of 
$\so_{2n-2}$.   Therefore generically, $e$ conjugates generate an irreducible
subalgebra of $\g$ and in particular, the center is central in $\g$.

Suppose that $n$ is odd.
On the irreducible module $X$ with highest weight the highest root,  there exist $e$ conjugates with
composition factors of dimensions $n^2-1$, $n(n-1)/2$,  $n(n-1)/2$
and also one where there is a composition factor of dimension
at least $(n-1)(2n-3)-1$.   Thus, generically there is a composition factor
of dimension at most $2n^2-5n+2$ and the smallest composition factor is at least
$n(n-1)/2$.  Since the sum of these two numbers (for $n \ge 9$) is greater than $\dim X = 2n^2 - n - 2$,
we see that generically $e$ conjugates acts irreducibly on $X$, whence
they generate $\g$  (by dimension).

Suppose that $n$ is even.   The same argument shows that $e$ conjugates
can generate a subalgebra having composition factors
on $[\g, \g]$  of dimensions $1$, $n^2-2$, $n(n-1)/2$, 
$n(n-1)/2$ and another $e$ conjugates having composition factors of dimensions
$1$,  $2n^2-5n+2$, $2n-2, 2n-2$.  This implies that generically $e$ conjugates act
irreducibly on $[\g,\g]/\z(\g)$ and this implies they generate $[\g,\g]$.

\medskip
Suppose that $x^{[2]} = 0$ and $n = 3$.  Then $x$ is the image of a square-zero element under the differential of $\SL_4 \to \SL_4/\mu_2 \cong \SO_6$, and $4$ conjugates of $x$ suffice to generate $[\g, \g]$ by Lemmas \ref{A.2.gen} and \ref{quo}\eqref{quo.der}.

\medskip
Suppose now that $x^{[2]} = x$.  If $\rank x \le n$, then let $y$ be the nilpotent element provided by Example \ref{D.2n.eg}, so $\rank y = \rank x$ and the claim follows from the nilpotent case.  

If $\rank x > n$, then set $x' = I_{2n} - x \in \so_{2n}$, which is toral of rank $2r \le n$.  Applying the previous case shows that $\max \{ 4, \lceil n / r \rceil \}$ conjugates of $x'$ generate a Lie subalgebra containing $[\g, \g]$.  Therefore, since $I_{2n}$ is central in $\so_{2n}$, the same number of conjugates of $x$ generate a Lie subalgebra containing $[\g, \g]$ by Lemma \ref{quo.gen}.
\end{proof}

\begin{eg} \label{D.2.root}
Suppose $x \in \so_{2n}$ satisfies $x^{[2]} = 0$, so the Jordan form of $x$ has $2r$ 2-by-2 blocks and $2n-4r$ 1-by-1 blocks for some $r \le n$.  There are two possibilities for the conjugacy class of $x$, see \cite[4.4]{Hesselink} or \cite[p.~70]{LiebeckSeitz}.  We focus on the larger class, the one where the restriction of the natural module to $x$ includes a 4-dimensional indecomposable denoted by $W_2(2)$ in \cite{LiebeckSeitz}.  The centralizer of such an $x$ in $\SO_{2n}$ has dimension
\[
\sum_{i=1}^{2r} 2(i-1) + \sum_{i=2r+1}^{2n-2r} (i - 1) = \binom{2n-2r}{2} + \binom{2r}{2},
\]
and therefore $\dim x^{\SO_{2n}} =  4r(n-r)$.  (The other class has dimension $2r(2n-2r-1)$.)
\end{eg}

\begin{cor}  \label{D.2}
Suppose $\car k = 2$.  For every noncentral $x \in \go_{2n}$ with $n \ge 4$ such that $x^{[2]} \in \{ 0, x \}$, there exist $e > 0$ and elements $x_1, \ldots, x_e \in x^{\SGO_{2n}}$ that generate a subalgebra containing $[ \so_{2n}, \so_{2n} ]$ such that $e \cdot \dim x^{\GO_{2n}} \le 4n^2$.
\end{cor}

\begin{proof} 
Suppose $x$ has $x^{[2]} = 0$ and rank $2r$ as in Example \ref{D.2.root}.  The condition we need is that $4n^2  \ge  e 4r(n-r)$.  
If the maximum in Prop.~\ref{D.2.gen} is 4, i.e., if $r \ge n/4$, then as a function of $r$, $16r(n-r)$ has a maximum of $4n^2$ at $r = n/2$.  Otherwise, the maximum is $e = \lceil n/r \rceil < (n+r)/r$, so $e \dim x^{\GO_{2n}} < 4(n^2 - r^2)$.  The right side has a maximum of $4 n^2 - 4$ at $r = 1$.

If $x^{[2]} = x$ and $x \in \so_{2n}$, the centralizer of $x$ in $\GO_{2n}$ has codimension 1 in $\GO_{2r'} \times \GO_{2(n-r')}$ when $x$ has rank $2r'$.  We may take $e = \max \{ 4, \lceil n/r \rceil \}$ where $2r$ is the dimension of the smallest eigenspace of $x$.  If $r' \ge n/4$, then $4 \dim x^{\GO_{2n}} \le 4n^2$ as for nilpotent elements.  So assume $r < n/4$. Then, as $r' = r$ or $n-r$, $e \dim x^{\GO_{2n}}$ is at most $(n+r')4r(n-r)/r' = 4(n^2 - s^2)$ for $s = r$ or $r'$, and again we conclude as for nilpotent elements.

If $x^{[2]} = x$ and $x \not\in \so_{2n}$, then $x$ is determined by choosing an ordered pair of ``parallel'' maximal isotropic subspaces and so the dimension of $x^{\GO_{2n}}$ agrees with the dimension of the flag variety of $D_n$ of parabolics with Levi subgroups of type $A_{n-2}$, which has dimension $(n^2+n-2)/2$.  Up to conjugacy, we may assume $x$ is the element from Example \ref{D.2n.eg}.  Let $y_0$ be an $n$-by-$n$ nilpotent matrix of with $\lfloor n/2 \rfloor$ 2-by-2 rank 1 Jordan blocks down the diagonal.  Then $y = \left( \begin{smallmatrix} 0 & y_0 \\ y_0 & 0 \end{smallmatrix} \right)$ is in $\so_{2n}$, is nilpotent, and 4 conjugates of $y$ suffice to generate a subalgebra containing $\so_{2n}$ (Prop.~\ref{D.2.gen}), so 4 conjugates of $x$ suffice as well.  As $2n^2 + 2n - 4 < 4n^2$, the claim is proved in this case.
\end{proof}

\section{Exceptional types}

The aim of this section is to provide the necessary material to prove Theorem \ref{MT} for exceptional groups, but we begin with some general-purpose observations.  Recall that a \emph{root element} of a Lie algebra $\g$ of $G$ is a generator for a one-dimensional root subalgebra $\g_\alpha$ of $\g$.  

\begin{lem} \label{root.1}
Let $G$ be a simple algebraic group such that $(G, \car k) \ne (\Sp_{2n}, 2)$ for $n \ge 2$.  For each nonzero nilpotent $x \in \g$, there is a root element in the closure of $x^G$.
\end{lem}

We ignore what happens in the excluded case.

\begin{proof}
Write $x = \sum_{\alpha \in S} X_\alpha$ where $S$ is a nonempty set of positive roots (relative to some torus $T$) and $X_\alpha$ is a generator for $\g_\alpha$.  If $|S| = 1$ (e.g., if $G$ has type $A_1$), then we are done.  Otherwise, the hypothesis on $(G, \car k)$ guarantees that no root vanishes on $T$, so we can pick a subtorus $T'$ of $T$ that centralizes some $X_\alpha$ but not some $X_{\alpha'}$ for some $\alpha \ne \alpha' \in S$.  Now in the closure of $x^{T'}$ we find a nonzero nilpotent supported on $S \setminus \{ \alpha' \}$, and by induction we are done.
\end{proof}

We say that a root element in $\g_\alpha$ is long (resp.~short) if $\alpha$ is long (resp.~short).

\begin{lem} \label{root.2}
Let $G$ be a simple linear algebraic group over a field $k$ such that $\car k$ is not special for $G$.  For every nonzero nilpotent $x \in \g$, there is a \underline{long} root element in the closure of $x^G$.
\end{lem}

\begin{proof}
By Lemma \ref{root.1}, we may assume that $G$ has two root lengths and that $x$ is a root element for a short root $\alpha$.

Suppose first that $G$ has rank 2, so $G$ has type $G_2$ and $\car k \ne 3$ or $G$ has type $C_2$ and $\car k \ne 2$.  Let $\alpha$ be the short simple root, $\gamma$ be the highest root (a long root), and take $\beta := \gamma - \alpha$.  Let $x_\alpha, x_\beta \!: \Ga \to G$ be the corresponding root subgroups.  These pick a generator $X_\alpha := \mathrm{d}x_\alpha(1)$ of $\g_\alpha$ such that 
\[
\ad(x_\beta(t)) X_\alpha = X_\alpha + N_{\beta,\alpha} X_\gamma,
\]
where $X_\gamma$ generates $\g_\gamma$, cf.~\cite[Ch.~3]{St}.  As $\car k$ is not special for $G$, $N_{\beta, \alpha}$ is not zero in $k$, and arguing as in the proof of Lemma \ref{root.1} we conclude that $k^\times X_\gamma$ meets the closure of $(X_\alpha)^G$, proving the claim in case $G$ has rank 2.

If $G$ has rank at least 3, pick a long root $\beta$ that is not orthogonal to $\alpha$ and let $G'$ be the subgroup of $G$ corresponding to the rank 2 sub-root-system generated by $\alpha$, $\beta$.  The ratio of the square-lengths of $\alpha$, $\beta$ is 2 so $G'$ has type $C_2$ and $\car k \ne 2$.  Then the closure of $x^{G'}$ contains a long root element in $G'$, hence in $G$.
\end{proof}

\begin{rmk}
Suppose that $G$ is a simple linear algebraic group over $k$ such that $\car k$ \underline{is} special for $G$.  The short root subalgebras generate a $G$-invariant subalgebra $\n$ of $\g$.  Omitting the case where $(G, \car k) = (\Sp_{2n}, 2)$ for $n \ge 2$, the same argument as in the proof of Lemma \ref{root.1} shows that for a nonzero nilpotent $x \in \g \setminus \n$ (resp., $\in \n$), there is a long (resp.~short) root element in the closure of $x^G$.
\end{rmk}

Now we focus on exceptional groups.  Table \ref{gen.classical} appears near Proposition \ref{gen.thm} below.

\begin{prop} \label{ex} 
Suppose  $G$ is simple of exceptional type over a field $k$ such that $\car k$ is not special for $G$.  For $e$ as in Table \ref{gen.classical}, $b(G)$ as in Table \ref{classical.table}, and $x \in \g$  noncentral, we have:
  \begin{enumerate}
\item \label{ex.gen} there are $x_1, \ldots, x_e \in x^G$ generating a Lie subalgebra of $\lie$ containing $[\lie, \lie]$, and
\item \label{ex.thm} $e \cdot \dim x^G \le b(G)$.
\end{enumerate}
\end{prop}

\begin{proof}
The crux is to prove \eqref{ex.gen}.  By taking closures as in Corollary \ref{gen.cor}, we may assume that the orbit $x^G$ of $x$
consists of root elements.
Moreover, as $k$ is not special, by Lemma \ref{root.2} we may assume that $x^G$ consists of long root elements.
In view of Lemma \ref{quo}\eqref{quo.der}, we may assume $\g$ is simply connected.

If $p \ne 2$, we can apply the result of \cite{CSUW}
to obtain \eqref{ex.gen}.    We now prove the result for $p=2$; in most cases, the argument also gives
another proof for all $p$.  

If $G=G_2$, we consider the $A_2$ subalgebra $\lsub$ generated by the long roots so $\lie/\lsub$ has the weights of $k^3 \oplus (k^3)^*$ as a representation of $\lsub$, so it is multiplicity free.  As $\lsub$ can be generated with 3 root elements (Prop.~\ref{A.2.gen}), the claim follows.

If $G = E_n$,  one uses that 4 root elements generate the $D_4$ inside $E_n$ (Example~\ref{D.2.comp}) and argue as for $G_2$, or one computes 
directly that five random root elements generate
$\lie$.   This completes the proof of \eqref{ex.gen}.

Claim \eqref{ex.thm} follows because 
\[
b(G) = e \cdot (\dim G - \rank G) \ge e \cdot \dim x^G. \qedhere
\]
\end{proof}

\section{Proof of Theorem \ref{MT}} 

\begin{lem} \label{MT.norep}
Let $G$ be a simple algebraic group over a field $k$ such that $p := \car k$ is not special.  Then for all noncentral $x \in \g$ such that $x^{[p]} \in \{ 0, x \}$, there exists $e>0$ and elements $x_1, \ldots, x_e \in x^G$ generating a subalgebra $\s$ of $\g$ containing $[\g, \g]$ such that $e \cdot \dim x^G \le b(G)$.
\end{lem}

\begin{proof}
Put $\pi \!: \Gt \to G$ for the simply connected cover of $G$.  If $\dpi \!: \gt \to \g$ is an isomorphism, then we apply \ref{A.odd} or \ref{A.2} for type $A$, \ref{BD.odd} for types $B$ or $D$ if $p \ne 2$, \ref{C.odd} for type $C$, and \ref{ex}\eqref{ex.thm} for the exceptional types.  If $G$ is adjoint of type $E_6$ and $\car k = 3$, we are done by Prop.~\ref{ex}.

Therefore, we may assume that $G = \SL_n / \mu_m$ and $p \mid m$, or $G$ has type $D_n$ and $p = 2$.  In these cases, \ref{A.odd}, \ref{A.2}, and \ref{D.2} concern not $G$ but a group $G'' := (G' \times \Gm)/Z(G')$ for some $G'$ isogenous to $G$.  In particular, putting $q \!: G' \to \Gb$ for the natural surjection onto the adjoint group, the induced map $\dq \!: \Lie(G'') \to \Lie(\Gb)$ is also a surjection.

Consider now the case $G = \Gb$.  Pick $y \in \Lie(G'')$ such that $\dq(y) = x$.  The results cited in the previous paragraph provide elements $y_1, \ldots, y_e \in y^{G''}$ such that $\s'' := \qform{ y_1, \ldots, y_e}$ contains $[\g'', \g'']$, and $e \cdot \dim y^{G''} \le b(G)$.  Taking $x_i := \dq(y_i)$, we obtain the desired result.

In the general case, write now $q$ for the natural map $G \to \Gb$.  For $z := \dq(x)$, let $z_1, \ldots, z_e \in z^{\Gb}$ by the elements provided by the adjoint case of the lemma.  Pick $g_i \in G(k)$ such that $z_i = \Ad(g_i) z$ and set $x_i := \Ad(g_i) x$.  Then $x_1, \ldots, x_e$ generate a subalgebra $\s$ such that $\dq(s) \supseteq [\gb, \gb]$.  Lemma \ref{quo.gen} completes the proof.
\end{proof}

We now prove the following result, which has the same hypotheses as Theorem \ref{MT} and a stronger conclusion.

\begin{thm} \label{MTp}
Let $G$ be a simple linear algebraic group over a field $k$ such that $p := \car k$ is not special for $G$.  If $\rho \!: G \to \GL(V)$ is a representation of $G$ such that $V$ has a $G$-subquotient $X$ with $X^{[\g, \g]} = 0$ and 
$\dim X > b(G)$ for $b(G)$ as in Table \ref{classical.table}, then $\dim x^G + \dim V^x < \dim V$ for all noncentral $x \in \g$ with $x^{[p]} \in \{ 0, x \}$.
\end{thm}

\begin{proof}
Assume for the moment that $V = X$.
We verify the inequality \eqref{ineq.mother} for the set $\mathcal{X}$ of noncentral $x \in \g$ such that $x^{[p]} \in \{ 0, x \}$.  Combining Lemma \ref{MT.norep} with \S\ref{key.sec} shows that \eqref{ineq.mother} holds for $x \in \mathcal{X}$.

For general $V$, it follows then that \eqref{ineq.mother} holds for $x \in \mathcal{X}$ as in Example \ref{ineq.sub}.
\end{proof}

\begin{proof}[Proof of Theorem \ref{MT}]
Combine Theorem \ref{MTp} with Lemmas \ref{ineq} and \ref{central}.
\end{proof}

\section{Small examples; proof of Corollary \ref{irred}} \label{irred.sec}

Before proving Corollary \ref{irred}, we provide an example that we treat in greater generality than is required for proving the corollary.  We put $\Sym^2 V$ for the second symmetric power of the vector space $V$.

\begin{lem} \label{SO.S2}
Suppose $\car k \ne 2$.   Let $G =  \SO(V)$ with $\dim V = n$.  Let  $W$
be the irreducible composition factor of $\Sym^2 V$ of dimension  $n(n+1)/2 -1$  if $\car k$ does not divide $n$, or $n(n+1)/2-2$
if $\car k$ divides $n$.    Then $\g$ acts generically freely on $V$.
\end{lem}

\begin{proof}
Let $S = \Sym^2 V$.  By fixing a basis for $V$, we may identify $S$ with $n$-by-$n$ symmetric matrices
and $\g$ with skew symmetric matrices.   Then  we see $W$ inside
$S$  (with $\g$ acting via Lie bracket in $\gl_n$).   

If $\car k$ does not divide $n$,  $W$ is just the trace zero matrices in $S$.   If $\car k$ divides $n$,
then $W$ is the set of trace zero matrices modulo scalars.

If we take an element of trace zero that is diagonal and generic, then
its centralizer in $\gl_n$ is just diagonal matrices (and even so for 
commuting modulo scalars).  Thus, its centralizer in $\g$ is 0, whence
the generic stabilizer in $\g$ is 0.
\end{proof}

\begin{proof}[Proof of Corollary \ref{irred}]
As $\car k$ is not special and we may assume that $\drho \ne 0$, $\ker \drho \subseteq \z(\g)$.  In case $\dim V < \dim G - \dim \z(\g)$, we have $\dim \drho(\g) v \le \dim V < \dim \drho(\g)$, whence $\g$ does not act virtually freely.

At the other extreme, if $\dim V > b(G)$ as in Table \ref{classical.table}, then $V$ is virtually free by Theorem \ref{MT} because $V^{[\g, \g]} = 0$ by Lemma \ref{quo}\eqref{quo.kill}.

For groups of classical type, the possible $V$ with $\dim V \le \dim G$ are listed in \cite[Table 2]{luebeck}.  The cases with $\dim G - \dim \z(\g) \le \dim V \le \dim G$ are settled in Lemma \ref{SO.S2} and Example \ref{adjoint}.

Consider first $G$ of type $A_\ell$.  By \cite[Th.~5.1]{luebeck}, either $\dim V \le \dim G$ or $\dim V > \ell^3/8$.  If $\ell \ge 20$, then $\ell^3/8 > b(G)$ and we are done.  For $16 \le \ell \le 19$, the tables in \cite{luebeck}\footnote{For $A_{18}$ and $A_{19}$, we refer to the extended table available on L\"ubeck's web page at \url{http://www.math.rwth-aachen.de/~Frank.Luebeck/chev/WMSmall/index.html}} show that there is no restricted dominant $\mu$ so that $\dim G < \dim L(\mu) \le b(G)$, completing the argument for type $A_\ell$.

For $G$ of type $B_\ell$, $C_\ell$, or $D_\ell$, the argument is the same but easier, with $\ell^3/8$ replaced by $\ell^3$.

Suppose now that $G$ has exceptional type.  The case $V = L(\hst)$ has been treated in Example \ref{adjoint}.  Otherwise,
Tables A.49--A.53 in L\"ubeck provide the following list of possibilities for $V$ with $b(G) \ge \dim V \ge \dim G - \dim \z(\g)$, up to graph automorphism and assuming $\car k$ is not special, where we denote the highest weights as in \cite{luebeck}:
$G_2$ with highest weight 02 and dimension 26 or 27 (where $\rho$ factors through $\SO_7$ and so is virtually free by Lemma \ref{SO.S2});
$G_2$ with highest weight 11 and dimension 38 and $\car k = 7$;
$F_4$ with highest weight 0010 and dimension 196 and $\car k = 3$;
$E_6$, with highest weight 000002 or 000010 and dimension 324 or 351.
These representations have $\dim V > \dim G$ and are virtually free by \cite[Th.~4.3.1]{Guerreiro}.  Note that for any particular $V$ and $\car k$, one can verify that the representation is virtually free using a computer, as described in \cite{GG:irred}.
\end{proof}

\section{How many conjugates are needed to generate $\Lie(G)$?} \label{gen.sec}

The results in the previous section suffice to prove the following, which generalizes the main result (Th.~8.2) of \cite{CSUW}.

\begin{prop} \label{gen.thm}
Let $G$ be a simple linear algebraic group over an algebraically closed field $k$ such that $\car k$ is not special for $G$, and let $e$ be as in Table \ref{gen.classical}.  If $x \in \g$ is noncentral, then there exist $e$ $G$-conjugates of $x$ that generate a subalgebra containing $[\g, \g]$.
 \end{prop}
 
 Recall that when $G$ is simply connected (and $\car k$ is not special), $\g = [\g, \g]$ as in Lemma \ref{quo}\eqref{quo.der}.

\begin{table}[hbt]
\begin{tabular}{cc||cc} 
type of $G$&$e$&type of $G$&$e$ \\ \hline
$A_\ell$ ($\ell \ge 1$) or $B_\ell$ ($\ell \ge 3$)&$\ell+1$&$G_2$&4 \\
$C_\ell$ ($\ell \ge 2$)&$2\ell$&$F_4$, $E_6$, $E_7$, $E_8$&5\\
$D_\ell$ ($\ell \ge 4$)&$\ell$
\end{tabular}
\caption{Number of conjugates $e$ needed to generate, as in Proposition \ref{gen.thm}.} \label{gen.classical}
\end{table}

The new results here are types $A$, $D$, $E$, and $G_2$ when $\car k = 2$.  The related result in \cite{CSUW} is stated for long root elements only, but the proof below shows
that the long root elements are the main case.

\begin{proof}
We first assume that $x$ is a long root element and $G$ is simply connected.   If $G$ is of exceptional type, we apply Proposition \ref{ex}, so assume that $G$ has type $A$, $B$, $C$, or $D$.
For type $A_n$, i.e., $G = \SL_{n+1}$, $n+1$ conjugates suffice by Proposition \ref{A.gen}\eqref{A.gen.o} if $\car k \ne 2$ and Proposition \ref{A.2.gen} if $\car k = 2$.
For type $C_n$ ($\Sp_{2n}$) with $n \ge 2$, $2n$ conjugates suffice by Proposition \ref{C.gen}\eqref{C.gen.o}.
For types $B$ and $D$, long root elements have rank 2 so Proposition \ref{BD.gen} gives the claim.  If $\car k = 2$ and $G$ has type $D_n$, then the claim follows for $\so_{2n}$ by Lemma \ref{D.2.rt}.  The claim follows for groups isogenous to $G$ by Lemma \ref{quo}.

If $x$ is nonzero nilpotent, then by Lemma \ref{root.2} and deforming as in \S\ref{deforming} we are reduced to the previous case.

Generally, $x$ has a Jordan decomposition $x = x_s + x_n$ where $x_s$ is semisimple and $x_n$ is nilpotent and we may assume $x_s \ne 0$.  If $x_s$ is noncentral, then we replace $x$ with $x_s$ (whose orbit is closed in the closure of $x^G$) and then replace $x_s$ with a root element as in Example \ref{deform.eg}.

Therefore, we may assume that $x_s, x_n \ne 0$ and $x_s$ is central.  Deforming, it suffices to treat the case where $x_n$ is a root element.  The line $tx_s + x_n$ for $t \in k$ has an open subset consisting of elements such that $e$ conjugates suffice to generate a subalgebra containing $[\g, \g]$, and this set is nonempty because it contains $x_n$, so it contains $t_0 x_s + x_n$ for some $t_0 \in \kx$.  The element $x_n$ and $t_0^{-1} x_n$ are in the same $G$-orbit, so the same is true of $x$ and $x_s + t_0^{-1} x_n$; this proves the claim.
\end{proof}

In the proof, the final paragraph could have been replaced by an argument that maps $x$ into the Lie algebra of the adjoint group of $G$ and applies the result for nilpotent elements there together with Lemma \ref{quo.gen}.

\section{The generic stabilizer in $G$ as a group scheme} \label{advert}

Let $G$ be an algebraic group over a field $k$ and $\rho \!: G \to \GL(V)$ a representation.  We say that $G$ acts \emph{generically freely} on $V$ if there is a dense open subset $U$ of $V$ such that, for every extension $K$ of $k$ and every $u \in U(K)$, the  stabilizer $G_u$ (a closed sub-group-scheme of $G \times K$) is the trivial group scheme 1.  Of course, $\ker \rho \subseteq G_u$ for all $u$, so it is natural to replace $G$ with $\rho(G)$ and assume that $G$ acts faithfully on $V$, i.e., $\ker \rho$ is the trivial group scheme.

In this section, we announce results on determining the generic stabilizer as a group scheme when $V$ is faithful and irreducible.  The proofs are combinations of the main results in this paper, the sequels \cite{GG:irred} and \cite{GG:special} (which build on this paper), and \cite{GurLawther}.

\begin{specthm} \label{MT.group}
Let $\rho \!: G \to \GL(V)$ be a faithful irreducible representation of a simple algebraic group over an algebraically closed field $k$.
\begin{enumerate}
\item \label{MT.group.et} $G_v$ is finite \'etale for generic $v \in V$ if and only if $\dim V > \dim G$ and $(G, V)$ does not appear in Table \ref{irred.nvfree}.
\item \label{MT.group.free} $G$ acts generically freely on $V$ if and only if $\dim V > \dim G$ and $(G, V)$ appears in neither Table \ref{irred.nvfree} nor Table \ref{ngenfree}.
\end{enumerate}
\end{specthm}

\begin{proof}
The stabilizer $G_v$ of a generic $v \in V$ is finite \'etale if and only if the stabilizer $\g_v$ of a generic $v \in V$ is zero, i.e., if and only if $\g$ acts generically freely on $V$.  By Theorem A in \cite{GG:irred}, this occurs if and only if $\dim V > \dim G$ and $(G, \car k, V)$ does not appear in Table \ref{irred.nvfree}, proving \eqref{MT.group.et}.

For \eqref{MT.group.free}, we must enumerate in Table \ref{ngenfree} those representations $V$ such that  $\dim V > \dim G$, $V$ does not appear in Table \ref{irred.nvfree}, and the group of points $G_v(k)$ is not trivial.  Those $V$ with the latter property are enumerated in \cite{GurLawther}, completing the proof.
\end{proof}

The results above settle completely the question of determining which faithful irreducible representations of simple $G$ are generically free.  It is natural to ask which of these hypotheses are necessary.  For example, if $\car k$ is special for $G$, there are irreducible but non-faithful representations that factor through the very special isogeny; whether or not these are virtually free for $\g$ is settled in \cite{GG:special}.  Another way that $G$ may fail to act faithfully is if $V$ is the Frobenius twist of a representation $V_0$; in that case $\g$ acts trivially on $V$, so $G$ acts virtually freely if and only if the group $G(k)$ of $k$-points acts virtually freely on $V_0$.  One could ask: \emph{What about analogues of the main results for $G$ semisimple?}  

One could also ask for a stronger bound in Theorem \ref{MT}.  What is the smallest constant $c$ such that the conclusion holds when we set $b(G) = c \dim G$?  What about to guarantee $G_v$ \'etale?  Or $G_v = 1$?  Table \ref{irred.nvfree} shows that $c$ must be greater than $1$.  Does $c = 2$ suffice?

\newpage
\begin{table}[bth]
{\centering\noindent\makebox[450pt]{
\begin{tabular}[c]{p{2.2in}|p{3in}}

${(A_\ell)~~}$
\begin{picture}(7,2)(0,0)
\put(0,1){\circle*{3}}
\put(0,1){\line(1,0){20}}
\put(20,1){\circle*{3}}
\put(20,1){\line(1,0){20}}
\put(40,1){\circle*{3}}
\put(40,-1.6){ \mbox{$\cdots$}}
\put(62,1){\circle*{3}}
\put(62,1){\line(1,0){20}}
\put(82,1){\circle*{3}}
\put(82,1){\line(1,0){20}}
\put(102,1){\circle*{3}}

\put(-2,-7){\mbox{\tiny $1$}}
\put(18,-7){\mbox{\tiny $2$}}
\put(38,-7){\mbox{\tiny $3$}}
\put(54,-7){\mbox{\tiny $\ell$$-$$2$}}
\put(75,-7){\mbox{\tiny $\ell$$-$$1$}}
\put(100,-7){\mbox{\tiny $\ell$}}
\end{picture}
\vspace{0.5cm}

&

${(C_\ell)~~}$
\begin{picture}(7,2)(0,0)
\put(0,1){\circle*{3}}
\put(0,0){\line(1,0){20}}
\put(0,2){\line(1,0){20}}
\put(20,1){\circle*{3}}
\put(20,1){\line(1,0){20}}
\put(40,1){\circle*{3}}
\put(40,-1.6){ \mbox{$\cdots$}}
\put(62,1){\circle*{3}}
\put(62,1){\line(1,0){20}}
\put(82,1){\circle*{3}}
\put(82,1){\line(1,0){20}}
\put(9,-1){{\small\mbox{$>$}}}
\put(102,1){\circle*{3}}

\put(-2,-7){\mbox{\tiny $1$}}
\put(18,-7){\mbox{\tiny $2$}}
\put(38,-7){\mbox{\tiny $3$}}
\put(54,-7){\mbox{\tiny $\ell$$-$$2$}}
\put(75,-7){\mbox{\tiny $\ell$$-$$1$}}
\put(100,-7){\mbox{\tiny $\ell$}}
\end{picture}
\vspace{0.5cm}
\\

${(B_\ell)~~}$
\begin{picture}(7,2)(0,0)
\put(0,1){\circle*{3}}
\put(0,0){\line(1,0){20}}
\put(0,2){\line(1,0){20}}
\put(20,1){\circle*{3}}
\put(20,1){\line(1,0){20}}
\put(40,1){\circle*{3}}
\put(40,-1.6){ \mbox{$\cdots$}}
\put(62,1){\circle*{3}}
\put(62,1){\line(1,0){20}}
\put(82,1){\circle*{3}}
\put(82,1){\line(1,0){20}}
\put(9,-1){{\small\mbox{$<$}}}
\put(102,1){\circle*{3}}

\put(-2,-7){\mbox{\tiny $1$}}
\put(18,-7){\mbox{\tiny $2$}}
\put(38,-7){\mbox{\tiny $3$}}
\put(54,-7){\mbox{\tiny $\ell$$-$$2$}}
\put(75,-7){\mbox{\tiny $\ell$$-$$1$}}
\put(100,-7){\mbox{\tiny $\ell$}}
\end{picture}
\vspace{0.5cm}

&

${(D_\ell)~~}$
\begin{picture}(7,2)(0,0)
\put(102,1){\circle*{3}}
\put(82,1){\line(1,0){20}}
\put(20,1){\circle*{3}}
\put(20,1){\line(1,0){20}}
\put(40,1){\circle*{3}}
\put(40,-1.6){ \mbox{$\cdots$}}
\put(62,1){\circle*{3}}
\put(62,1){\line(1,0){20}}
\put(82,1){\circle*{3}}
\put(20,2){\line(-4,3){15}}
\put(20,0){\line(-4,-3){15}}
\put(5,12.9){\circle*{3}}
\put(5,-10.9){\circle*{3}}

\put(-2,-12){\mbox{\tiny $2$}}
\put(-2,11.8){\mbox{\tiny $1$}}
\put(18,-7){\mbox{\tiny $3$}}
\put(38,-7){\mbox{\tiny $4$}}
\put(54,-7){\mbox{\tiny $\ell$$-$$2$}}
\put(75,-7){\mbox{\tiny $\ell$$-$$1$}}
\put(100,-7){\mbox{\tiny $\ell$}}
\end{picture}

\end{tabular}
}}
\caption{Dynkin diagrams of simple root systems of classical type, with simple roots numbered as in \cite{luebeck}.} \label{luebeck.table}
\end{table}

\begin{table}[htbp]
\begin{tabular}{cccrc||cccrc}
$G$&$\car k$&rep'n&$\dim V$&$\dim \g_v$&$G$&$\car k$&high weight&$\dim V$&$\dim \g_v$\\ \hline\hline
$\SL_8/\mu_4$&2&$\wedge^4$&70&3&$\Sp_8$&3&0100&40&2 \\
$\SL_9/\mu_3$&3&$\wedge^3$&84&2&$\Sp_4$&5&11&12&1\\
$\Spin_{16}/\mu_2$&2&half-spin&128&4&$\SL_4$&$p$ odd&$01p^e$, $e \ge 1$&24&1 \\
&&&&&$\SL_4/\mu_2$&2&$012^e$, $e \ge 2$&24&1
\end{tabular}
\caption{Irreducible and faithful representations $V$ with restricted highest weight of simple $G$ with $\dim V > \dim G$ that are not generically free for $\g$, up to graph automorphism.  For each, the stabilizer $\g_v$ of a generic $v \in V$ is a toral subalgebra.  The weights on the right side are numbered as in Table \ref{luebeck.table}.} \label{irred.nvfree}
\end{table}

\begin{table}[htbp]
\begin{tabular}{cccr||cccr}
$G$&$\car k$&$V$&$\dim V$&$G$&$\car k$&$V$&$\dim V$\\ \hline\hline
$A_1$&$\ne 2, 3$&$S^3$&4&$A_2$&$\ne 2, 3$&$S^3$&10 \\
$A_1$&$\ne 2, 3$&$S^4$&5&$A_3$&$\ne 2$&$L(2\omega_2)$&19 or 20 \\
$A_8$&$\ne 3$&$\wedge^3$&84&$A_7$&$\ne 2$&$\wedge^4$&70 \\
$A_3$&3&$L(\omega_1+\omega_2)$&16&$A_\ell$&$p \ne 0$&$L(\omega_1 + p^i \omega_\ell)$,&$(\ell+1)^2$ \\
&&&&&&$L(\omega_1 + p^i \omega_1)$ \\ 
$B_\ell$ ($\ell \ge 2$)&$\ne 2$&$L(2\omega_\ell)$& $2\ell^2 - 3\ell - \eps$&$C_4$&$\ne 2$&``spin''&41 or 42 \\
$D_\ell$ ($\ell \ge 4$)&$\ne 2$&$L(2\omega_\ell)$&$2\ell^2 + \ell -1 - \eps$ &$D_8$&$\ne 2$&half-spin&128
\end{tabular}
\caption{Irreducible faithful representations $V$ of simple $G$ with $\dim V > \dim G$ such that $G_v$ is finite \'etale and $\ne 1$ for generic $v \in V$, up to graph automorphism, adapted from \cite{GurLawther}.  The symbol $\eps$ denotes 0 or 1 depending on the value of $\car k$.} \label{ngenfree}
\end{table}

\clearpage

\bibliographystyle{amsalpha}
\providecommand{\bysame}{\leavevmode\hbox to3em{\hrulefill}\thinspace}
\providecommand{\MR}{\relax\ifhmode\unskip\space\fi MR }
\providecommand{\MRhref}[2]{%
  \href{http://www.ams.org/mathscinet-getitem?mr=#1}{#2}
}
\providecommand{\href}[2]{#2}

\end{document}